\documentclass[a4paper , 11pt]{amsart}
\usepackage[top=30truemm,bottom=35truemm,left=30truemm,right=30truemm]{geometry}
\usepackage{amsmath}
\usepackage{amsthm}
\usepackage{amsfonts}
\usepackage{bm}
\newtheorem{thm}{Theorem}[section]
\newtheorem{lem}[thm]{Lemma}
\newtheorem{prop}[thm]{Proposition}
\newtheorem{cor}[thm]{Corollary}
\theoremstyle{definition}
\newtheorem{prob}[thm]{Problem}
\newtheorem*{rem}{\rm Remark}

\DeclareMathOperator{\insep}{insep}
\DeclareMathOperator{\Dio}{Dio}
\DeclareMathOperator{\Res}{Res}
\DeclareMathOperator{\Disc}{Disc}
\DeclareMathOperator{\Card}{Card}
\makeatletter 
\@tempcnta\z@
\loop\ifnum\@tempcnta<26
\advance\@tempcnta\@ne
\expandafter\edef\csname f\@Alph\@tempcnta\endcsname{\noexpand\mathfrak{\@Alph\@tempcnta}}
\expandafter\edef\csname l\@Alph\@tempcnta\endcsname{\noexpand\mathbb{\@Alph\@tempcnta}}
\expandafter\edef\csname c\@Alph\@tempcnta\endcsname{\noexpand\mathcal{\@Alph\@tempcnta}}
\repeat
\title{Quadratic approximation in $\lF _q((T^{-1}))$}
\author{Tomohiro Ooto}
\date{}
\begin{document}
\address{Graduate School of Pure and Applied Sciences, University of Tsukuba, Tennodai 1-1-1, Tsukuba, Ibaraki, 305-8571, Japan}
\email{ooto@math.tsukuba.ac.jp}

\begin{abstract}
In this paper, we study Diophantine exponents $w_n$ and $w_n ^{*}$ for Laurent series over a finite field.
Especially, we deal with the case $n=2$, that is, quadratic approximation.
We first show that the range of the function $w_2-w_2 ^{*}$ is exactly the closed interval $[0,1]$.
Next, we estimate an upper bound of the exponent $w_2$ of continued fractions with low complexity partial quotients.
\end{abstract}

\maketitle

\tableofcontents

\section{Introduction}
Let $p$ be a prime and $q$ be a power of $p$.
We denote by $\lF _q$ the finite field of elements $q$, $\lF _q[T]$ the ring of polynomials over $\lF _q$, $\lF _q(T)$ the field of rational functions over $\lF _q$, and $\lF _q((T^{-1}))$ the field of Laurent series over $\lF _q$.
For $\xi \in \lF _q((T^{-1})) \setminus \{ 0\}$, we can write
\begin{eqnarray*}
\xi = \sum _{n=N}^{\infty }a_n T^{-n},
\end{eqnarray*}
where $N\in \lZ $, $a_n \in \lF _q$, and $a_N \neq 0$.
We define the absolute value on $\lF _q ((T^{-1}))$ by $|0|:=0$ and $|\xi |:=q^{-N}$.
The absolute value can uniquely extend on the algebraic closure of $\lF _q((T^{-1}))$ and we continue to write $|\cdot |$ for the extended absolute value.

Throughout this paper, we regard elements of $(\lF _q[T])[X]$ as polynomials in $X$.
For $P(X) \in (\lF _q [T])[X]$, the {\itshape height} of $P(X)$, denoted by $H(P)$, is defined to be the maximal of absolute values of the coefficients of $P(X)$.
We denote by $(\lF _q[T])[X]_{\min }$ the set of non-constant irreducible primitive polynomials $P(X)\in (\lF _q[T])[X]$ whose the leading coefficient is monic in $T$.
For $\alpha \in \overline{\lF _q(T)}$, there exists a unique polynomial $P(X)\in (\lF _q[T])[X]_{\min }$ such that $P(\alpha )=0$. 
We call $P(X)$ the {\itshape minimal polynomial} of $\alpha $.
The {\itshape height} (resp.\ the {\itshape degree}, the {\itshape inseparable degree}) of $\alpha $, denoted by $H(\alpha )$ (resp.\ $\deg \alpha $, $\insep \alpha $), is defined to be the height of $P(X)$ (resp.\ the degree of $P(X)$, the inseparable degree of $P(X)$).
For $\xi \in \lF _q((T^{-1}))$ and integers $n, H\geq 1$, let $w_n (\xi ,H)$ and $w_n ^{*}(\xi ,H)$ be given by
\begin{gather*}
w_n (\xi ,H)=\min \{ |P(\xi )| \mid P(X)\in (\lF _q[T])[X], H(P)\leq H, \deg _X P\leq n, P(\xi )\neq 0 \} , \\
w_n ^{*}(\xi ,H)=\min \{ |\xi -\alpha | \mid \alpha \in \overline{\lF _q(T)}, H(\alpha )\leq H, \deg \alpha \leq n, \alpha \neq \xi \} . 
\end{gather*}
The Diophantine exponents $w_n$ and $w_n ^{*}$ are defined by
\begin{gather*}
w_n(\xi )=\limsup _{H\rightarrow \infty } \frac{-\log w_n(\xi ,H)}{\log H},\quad w_n ^{*}(\xi )=\limsup _{H\rightarrow \infty } \frac{-\log H w_n ^{*}(\xi ,H)}{\log H}.
\end{gather*}
In other words, $w_n(\xi )$ (resp.\ $w_n ^{*}(\xi )$) is the supremum of a real number $w$ (resp.\ $w^{*}$) which is satisfied that
\begin{eqnarray*}
0<|P(\xi )|\leq H(P)^{-w}\quad (\mbox{resp.\ } 0<|\xi -\alpha |\leq H(\alpha )^{-w^{*}-1})
\end{eqnarray*}
for infinitely many $P(X)\in (\lF _q[T])[X]$ of degree at most $n$ (resp.\ $\alpha \in \overline{\lF _q(T)}$ of degree at most $n$).

As in classical continued fraction theory of real numbers, if $\xi \in \lF _q((T^{-1}))$, then we can write
\begin{eqnarray*}
\xi = a_0+\cfrac{1}{a_1+\cfrac{1}{a_2+\cfrac{1}{\cdots }}},
\end{eqnarray*}
where $a_0, a_n \in \lF _q[T]$, $\deg a_n\geq 1$ for $n\geq 1$.
For simplicity of notation, we write $\xi =[a_0,a_1,a_2,\ldots ]$.
The $a_0 $ and $a_n $ are called the {\itshape partial quotients} of $\xi $.
We define $p_n$ and $q_n$ by
\begin{eqnarray*}
\begin{cases}
p_{-1}=1,\ p_0=a_0,\ p_n=a_n p_{n-1}+p_{n-2},\ n\geq 1,\\
q_{-1}=0,\ q_0=1,\ q_n=a_n q_{n-1}+q_{n-2},\ n\geq 1.
\end{cases}
\end{eqnarray*}
We call $(p_n/q_n)_{n\geq 0}$ {\itshape convergent sequence} of $\xi $ and have $p_n/q_n=[a_0,a_1,\ldots ,a_n]$ for $n\geq 0$ by induction on $n$.

In this paper, we study the difference of the Diophantine exponents $w_2$ and $w_2 ^{*}$ using continued fractions.
We denote by $\lfloor x\rfloor $ the integer part and $\lceil x\rceil $ the upper integer part of a real number $x$.
We construct explicitly continued fractions $\xi \in \lF _q((T^{-1}))$ for which $w_2(\xi )-w_2 ^{*}(\xi )=\delta $ for each $0<\delta \leq 1$ as follows:

\begin{thm}\label{main4}
Let $w$ be a real number which is greater than $(5+\sqrt{17})/2$ when $p\neq 2$, and that of $(9+\sqrt{65})/2$ when $p=2$.
Let $b,c\in \lF _q[T]$ be distinct polynomials of degree at least one.
We define a sequence $(a_{n,w})_{n\geq 1}$ by
\begin{eqnarray*}
a_{n,w}=\begin{cases}
c & \mbox{if } n=\lfloor w^i \rfloor \mbox{ for some integer } i\geq 0,\\
b & \mbox{otherwise}.
\end{cases}
\end{eqnarray*}
Set $\xi _w :=[0,a_{1,w},a_{2,w},\ldots ]$.
Then we have $w_2 ^{*}(\xi _w)=w-1$ and $w_2(\xi _w)=w$. 
\end{thm}

\begin{thm}\label{main5}
Let $w\geq 25$ be a real number and $b,c,d\in \lF _q[T]$ be distinct polynomials of degree at least one.
Let $0<\eta <\sqrt{w}/4$ be a positive number and put
\begin{eqnarray*}
m_i :=\left\lfloor \frac{\lfloor w^{i+1}\rfloor -\lfloor w^i-1\rfloor }{\lfloor \eta w^i\rfloor } \right\rfloor 
\end{eqnarray*}
for all $i\geq 1$.
We define a sequence $(a_{n,w,\eta })_{n\geq 1}$ by
\begin{eqnarray*}
a_{n,w,\eta }=\begin{cases}
c & \mbox{if } n=\lfloor w^i \rfloor \mbox{ for some integer } i\geq 0,\\
d & \parbox{250pt}{$\mbox{if } n \neq \lfloor w^i\rfloor \mbox{ for all integer }i\geq 0 \mbox{ and } n=\lfloor w^j\rfloor +m \lfloor \eta w^j\rfloor \mbox{ for some integer }1\leq m\leq m_j, j\geq 1,$}\\
b & \mbox{otherwise}.
\end{cases}
\end{eqnarray*}
Set $\xi _{w,\eta } :=[0,a_{1,w,\eta },a_{2,w,\eta },\ldots ]$.
Then we have
\begin{align*}
w_2 ^{*}(\xi _{w,\eta })=\frac{2 w-2-\eta }{2+\eta },\quad w_2(\xi _{w,\eta })=\frac{2 w-\eta }{2+\eta }.
\end{align*}
Hence, we have
\begin{eqnarray*}
w_2(\xi _{w,\eta })-w_2 ^{*}(\xi _{w,\eta })=\frac{2}{2+\eta }.
\end{eqnarray*}
\end{thm}

Theorem \ref{main4}, \ref{main5} are analogues of Theorem 4.1, 4.2 in \cite{Bugeaud2} and Theorem 1, 2 in \cite{Bugeaud3}.
Theorem \ref{main4}, \ref{main5} are proved in a similar method of the proof of these analogue theorems.

In Section \ref{Properties sec}, we prove
\begin{eqnarray*}
0\leq w_2(\xi )-w_2 ^{*}(\xi )\leq 1
\end{eqnarray*}
for all $\xi \in \lF _q((T^{-1}))$ (Proposition \ref{main7}).
We also prove $w_n (\xi )=w_n ^{*}(\xi )=0$ for all $n\geq 1$ and $\xi \in \lF _q(T)$ (Theorem \ref{alg}).
Consequently, we determine the range of the function $w_2-w_2 ^{*}$ from Theorem \ref{main4} and \ref{main5}.

\begin{cor}\label{main6}
The range of the function $w_2-w_2 ^{*}$ is exactly the closed interval $[0,1]$.
\end{cor}

For $\xi \in \lF _q((T^{-1}))$, we set
\begin{eqnarray*}
w(\xi ) :=\limsup _{n\rightarrow \infty }\frac{w_n(\xi )}{n},\quad w^{*}(\xi ) :=\limsup _{n\rightarrow \infty }\frac{w_n ^{*}(\xi )}{n}.
\end{eqnarray*}
We say that $\xi $ is an
\begin{gather*}
A\mbox{-{\itshape number} if } w(\xi )=0;\\
S\mbox{-{\itshape number} if } 0<w(\xi )<+\infty ;\\
T\mbox{-{\itshape number} if } w(\xi )=+\infty \mbox{ and } w_n(\xi )<+\infty \mbox{ for all } n;\\
U\mbox{-{\itshape number} if } w(\xi )=+\infty \mbox{ and } w_n(\xi )=+\infty \mbox{ for some } n.
\end{gather*}
This classification of $\lF _q((T^{-1}))$ was first introduced by Bundschuh \cite{Bundschuh} and is called {\itshape Mahler's classification}.
Replacing $w_n$ and $w$ with $w_n ^{*}$ and $w^{*}$, we define $A^{*}$-, $S^{*}$-, $T^{*}$-, and $U^{*}$-numbers as above.
This classification of $\lF _q((T^{-1}))$ was first introduced by Bugeaud \cite[Section 9]{Bugeaud1} and is called {\itshape Koksma's classification}.
Let $n\geq 1$ be an integer, $\xi \in \lF _q((T^{-1}))$ be a $U$-number, and $\zeta \in \lF _q((T^{-1}))$ be a $U^{*}$-number.
The number $\xi $ (resp.\ the number $\zeta $) is called a $U_n$-{\itshape number} (resp.\ $U_n ^{*}$-{\itshape number}) if $w_n(\xi )$ is infinite and $w_m(\xi )$ is finite (resp.\ $w_n ^{*}(\zeta )$ is infinite and $w_m ^{*}(\zeta )$ is finite) for all $1\leq m<n$.

Let $\mathcal{A}$ be a finite set.
Let $\mathcal{A}^{*}$, $\mathcal{A}^{+}$, and $\mathcal{A}^{\lN }$ denote the set of finite words over $\mathcal{A}$,  the set of nonempty finite words over $\mathcal{A}$, and the set of infinite words over $\mathcal{A}$.
We denote by $|W|$ the length of a finite word $W$ over $\mathcal{A}$.
For an integer $n\geq 0$, let $W^n=WW\cdots W$ ($n$ times repeated concatenation of the word $W$) and $\overline{W}=WW\cdots W\cdots $ (infinitely many concatenation of the word $W$).
Note that $W^0$ is equal to the empty word.
More generally, for a real number $w\geq 0$, let $W^{w}=W^{\lfloor w\rfloor}W'$, where $W'$ is the prefix of $W$ of length $\lceil (w-\lfloor w\rfloor )|W|\rceil $.
Let ${\bf a}=(a_n)_{n\geq 0}$ be a sequence over $\mathcal{A}$.
We identify ${\bf a}$ with the infinite word $a_0 a_1 \cdots a_n \cdots $.
Let $\rho $ be a real number.
We say that ${\bf a}$ satisfies {\itshape Condition} $(*)_{\rho }$ if there exist sequences of finite words $(U_n)_{n\geq 1}$, $(V_n)_{n\geq 1}$ and a sequence of nonnegative real numbers $(w_n)_{n\geq 1}$ such that
\begin{enumerate}
 \item[(i)] the word $U_n V_n ^{w_n}$ is the prefix of ${\bf a}$ for all $n\geq 1$,
 \item[(ii)] $|U_n V_n ^{w_n}|/|U_n V_n|\geq \rho $ for all $n\geq 1$,
 \item[(iii)] the sequence $(|V_n ^{w_n}|)_{n\geq 1}$ is strictly increasing.
\end{enumerate}
The {\itshape Diophantine exponent} of ${\bf a}$, first introduced in \cite{Adamczewski2}, denoted by $\Dio ({\bf a})$, and is defined to be the supremum of a real number $\rho $ such that ${\bf a}$ satisfy Condition $(*)_{\rho }$.
It is obvious that
\begin{eqnarray*}
1 \leq \Dio ({\bf a}) \leq +\infty .
\end{eqnarray*}
The infinite word ${\bf a}$ is called {\itshape ultimately periodic} if there exist finite words $U\in \mathcal{A}^{*}$ and $V\in \mathcal{A}^{+}$ such that ${\bf a}=U\overline{V}$.
The {\itshape complexity function} of ${\bf a}$ is defined by
\begin{eqnarray*}
p({\bf a}, n)= \Card \{a_i a_{i+1}\ldots a_{i+n-1}\mid i\geq 0 \} ,\quad \mbox{for }\ n\geq 1.
\end{eqnarray*}

We state now the second main results.

\begin{thm}\label{main1}
Let $\kappa \geq 2, A\geq q$ be integers and ${\bf a}=(a_n)_{n\geq 1}$ be a sequence over $\lF _q[T]$ with $q\leq |a_n|\leq A$ for all $n \geq 1$.
Assume that there exists an integer $n_0\geq 1$ such that
\begin{eqnarray*}
p({\bf a}, n) \leq \kappa n \mbox{ for all } n\geq n_0,
\end{eqnarray*}
and the Diophantine exponent of ${\bf a}$ is finite.
Set $\xi :=[0,a_1,a_2,\ldots ]$.
Then we have
\begin{eqnarray}\label{main1.1}
w_2 (\xi ) \leq 128(2\kappa +1)^3 \Dio ({\bf a}) \left( \frac{\log A}{\log q}\right) ^4 .
\end{eqnarray}
In particular, if the sequence $(|q_n|^{1/n})_{n\geq 1}$ converges, then we have
\begin{eqnarray}\label{main1.2}
w_2 (\xi ) \leq 64(2\kappa +1)^3 \Dio ({\bf a}).
\end{eqnarray}
\end{thm}

There are special sequences which are satisfied the assumption of Theorem \ref{main1}, for example, automatic sequences,  primitive morphic sequences, and Strumian sequence with some condition.
The detail will appear in Section $2$.

\begin{thm}\label{main2}
Let ${\bf a}=(a_n)_{n\geq 1}$ be a non-ultimately periodic sequence over $\lF _q[T]$ with $\deg a_n \geq 1$ for all $n\geq 1$.
Assume that $(|q_n |^{1/n})_{n\geq 1}$ is bounded.
Put
\begin{eqnarray*}
m:= \liminf_{n\rightarrow \infty } |q_n|^{1/n},\quad M:=\limsup_{n\rightarrow \infty} |q_n|^{1/n}.
\end{eqnarray*}
Set $\xi :=[0,a_1,a_2,\ldots ]$.
Then we have
\begin{eqnarray}\label{lDio}
w_2(\xi )\geq w_2 ^{*}(\xi ) \geq \max \left( 2,\frac{\log m}{\log M} \Dio ({\bf a})-1\right).
\end{eqnarray}
In particular, if the sequence $(|q_n|^{1/n})_{n\geq 1}$ converges, then we have
\begin{eqnarray*}
w_2(\xi )\geq w_2 ^{*}(\xi ) \geq \max (2,\Dio ({\bf a})-1).
\end{eqnarray*}
Furthermore, assume that the sequence $(|a_n|)_{n\geq 1}$ is bounded.
Then we have
\begin{eqnarray}\label{lDio2}
w_2(\xi )\geq \max \left( 2,\frac{\log m}{\log M}( \Dio ({\bf a})+1)-1\right).
\end{eqnarray}
In particular, if the sequence $(|q_n|^{1/n})_{n\geq 1}$ converges, then we have
\begin{eqnarray*}
w_2(\xi )\geq \max (2,\Dio ({\bf a})).
\end{eqnarray*}
\end{thm}

Theorem \ref{main1} and \ref{main2} are analogues of Theorem 2.2 and 2.3 in \cite{Bugeaud2}.
Theorem \ref{main1} and \ref{main2} are proved in a similar method of the proof of these analogue theorems.

We state an immediately consequence of Theorem \ref{main1} and \ref{main2}.

\begin{cor}\label{main3}
Let ${\bf a}=(a_n)_{n\geq 1}$ be a non-ultimately periodic sequence over $\lF _q[T]$ with $\deg a_n\geq 1$ for $n\geq 1$.
Assume that $(|a_n|)_{n\geq 1}$ is bounded and
\begin{eqnarray*}
\limsup_{n\rightarrow \infty } \frac{p({\bf a}, n)}{n}<+\infty .
\end{eqnarray*}
Set $\xi :=[0,a_1, a_2,\ldots ]$.
Then the Diophantine exponent of ${\bf a}$ is finite if and only if $\xi $ is not a $U_2$-number.
\end{cor}

We use the Vinogradov notation $A\ll B$ (resp.\ $A\ll _a B$) if $|A|\leq c |B|$ with some constant (resp.\ some constant depending at most on $a$) $c>0$.
We write $A\asymp B$ (resp.\ $A\asymp _a B$) if $A\ll B$ and $B\ll A$ (resp.\ $A\ll _a B$ and $B\ll _a A$) hold.

This paper is organized as follows.
In Section \ref{Automatic sec}, we define special sequences and apply Theorem \ref{main1} to these sequences.
In Section \ref{Liouville sec}, we prove Liouville inequalities, that is, a nontrivial lower bound of the absolute value of difference two algebraic numbers and that of polynomial at an algebraic point.
In Section \ref{Continued sec}, we prove some lemmas with respect to continued fractions.
In Section \ref{Properties sec}, we study the Diophantine exponents $w_n$ and $w_n ^{*}$.
In Section \ref{Applications sec}, applying Liouville inequality, we prove lemmas to determine the value of $w_2$ and $w_2 ^{*}$.
In Section \ref{Combinational sec}, we prove combinational lemma to show Theorem \ref{main1}.
In Section \ref{Proof sec}, we prove Theorem \ref{main4}, \ref{main5}, \ref{main1}, and \ref{main2}.
In Appendix \ref{Rational sec}, we prove an analogue of Theorem \ref{main1} and \ref{main2} for Laurent series over a finite field.

\section{Application of the main results}\label{Automatic sec}

In this section, we first recall properties of special sequences.
For a deeper discussion, we refer the reader to \cite{Allouche}.
Let $k\geq 2$ be an integer.
We denote by $\Sigma _k$ the set $\{ 0,1,\ldots ,k-1 \}$.
A $k$-{\itshape automaton} is a sextuple
\begin{eqnarray*}
A=(Q, \Sigma _k, \delta , q_0, \Delta ,\tau ),
\end{eqnarray*} 
where $Q$ is a finite set, $\delta :Q\times \Sigma _k\rightarrow Q$ is a map, $q_0 \in Q$, $\Delta $ is a set, and $\tau :Q\rightarrow \Delta $ is a map.
For $q\in Q$ and a finite word $W=w_0 w_1 \cdots w_m$ over $\Sigma _k$, we define recursively $\delta (q,W)$ by $\delta (q,W)=\delta (\delta (q,w_0 w_1\cdots w_{m-1}), w_m)$.
Let $n\geq 0$ be an integer and $W_n=w_r w_{r-1}\cdots w_0$, where $\sum _{i=0}^{r}w_i k^i$ is the $k$-ary expansion of $n$.
A sequence ${\bf a}=(a_n)_{n\geq 0}$ is said to be $k$-{\itshape automatic} if there exists a $k$-automaton $A=(Q, \Sigma _k, \delta , q_0, \Delta ,\tau )$ such that $a_n=\tau (\delta (q_0,W_n))$ for all $n\geq 0$.
The $k$-{\itshape kernel} of a sequence ${\bf a}=(a_n)_{n\geq 0}$ is the set of all sequences $(a_{k^i n+j})_{n\geq 0}$, where $i\geq 0$ and $0\leq j<k^i$.

It is known about $k$-automatic sequences as follows:

\begin{thm}[Eilenberg \cite{Eilenberg}]
Let $k\geq 2$ be an integer.
Then a sequence is $k$-automatic if and only if its $k$-kernel is finite.
\end{thm}

\begin{lem}[Adamczewski and Cassaigne \cite{Adamczewski1}]
Let $k\geq 2$ be an integer.
Let ${\bf a}$ be a non-ultimately periodic and $k$-automatic sequence.
Let $m$ be a cardinality of the $k$-kernel of ${\bf a}$.
Then we have
\begin{eqnarray*}
\Dio ({\bf a})<k^m.
\end{eqnarray*}
\end{lem}

Let $\mathcal{A}$ and $\mathcal{B}$ be finite sets.
A map $\sigma :\mathcal{A}^{*}\rightarrow \mathcal{B}^{*}$ is a {\itshape morphism} if $\sigma (UV)=\sigma (U)\sigma (V)$ for all $U,V \in \mathcal{A}$.
The {\itshape width} of $\sigma $ is defined to be $\max _{a\in \mathcal{A}}|\sigma (a)|$.
A morphism $\sigma $ is said to be $k$-{\itshape uniform} if there exists an integer $k\geq 1$ such that $|\sigma (a)|=k$ for all $a\in \mathcal{A}$.
In particular, we call a $1$-uniform morphism a {\itshape coding}.
A morphism $\sigma :\mathcal{A}^{*}\rightarrow \mathcal{A}^{*}$ is {\itshape primitive} if there exists an integer $n\geq 1$ such that $a$ occurs in $\sigma ^n(b)$ for all $a,b \in \mathcal{A}$.
A morphism $\sigma :\mathcal{A}^{*}\rightarrow \mathcal{A}^{*}$ is {\itshape prolongable} on $a\in \mathcal{A}$ if $\sigma (a)=aW,$ where $W\in \mathcal{A}^{+}$ and $\sigma ^n(W)$ is not an empty word for all $n\geq 1$.
A sequence ${\bf a}=(a_n)_{n\geq 0}$ is said to be $k$-{\itshape uniform morphic} (resp.\ {\itshape primitive morphic}) if there exist finite sets $\mathcal{A},\mathcal{B}$, a $k$-uniform morphism (resp.\ a primitive morphism) $\sigma :\mathcal{A}^{*}\rightarrow \mathcal{A}^{*}$ which is prolongable on some $a \in\mathcal{A}$, and a coding $\tau :\mathcal{A}^{*}\rightarrow \mathcal{B}^{*}$ such that ${\bf a}=\lim_{n\rightarrow \infty }\tau (\sigma ^n(a))$.
When ${\bf a}$ is a $k$-uniform morphic, we call $\mathcal{A}$ the {\itshape initial alphabet} associated with ${\bf a}$.

\begin{thm}[Cobham \cite{Cobham}]
Let $k\geq 2$ be an integer.
Then a sequence is $k$-automatic if and only if it is $k$-uniform morphic.
\end{thm}

Moss\'e's result \cite{Mosse} implies the lemma below.

\begin{lem}
Let ${\bf a}$ be a non-ultimately periodic and primitive morphic sequence.
Then the Diophantine exponent of ${\bf a}$ is finite.
\end{lem}

Let $\theta $ and $\rho $ be real numbers with $0<\theta <1$ and $\theta $ is irrational.
For $n\geq 1$, we put $s_{n,\theta ,\rho }:=\lfloor (n+1)\theta +\rho \rfloor - \lfloor n\theta +\rho \rfloor $ and $s_{n, \theta ,\rho } ':=\lceil (n+1)\theta +\rho \rceil -\lceil n\theta +\rho \rceil $.
A sequence ${\bf a}=(a_n)_{n\geq 1}$ is called {\itshape Sturmian} if there exist an irrational number $0<\theta <1$, a real number $\rho $, a finite set $\mathcal{A}$, and a coding $\tau :\{ 0,1\} ^{*}\rightarrow \mathcal{A}^{*}$ with $\tau (0)\neq \tau (1)$ such that $(a_n)_{n\geq 1}$ is $(\tau (s_{n,\theta ,\rho }))_{n\geq 1}$ or $(\tau (s_{n,\theta ,\rho } '))_{n\geq 1}$.
Then we call the irrational number $\theta $ {\itshape slope} of ${\bf a}$ and the real number $\rho $ {\itshape intercept} of ${\bf a}$.

\begin{lem}[Adamczewski and Bugeaud \cite{Adamczewski4}]
Let ${\bf a}$ be a Strumian sequence.
Then the slope of ${\bf a}$ has bounded partial quotients if and only if the Diophantine exponent of ${\bf a}$ is finite.
\end{lem}

It is known that automatic sequences, primitive morphic sequences, and Strumian sequences have low complexity.

\begin{lem}
Let $k\geq 2$ be an integer and ${\bf a}=(a_n)_{n\geq 0}$ be a $k$-automatic sequence.
Let $d$ be a cardinality of the internal alphabet associated with ${\bf a}$.
Then we have
\begin{eqnarray*}
p({\bf a}, n) \leq k d^2 n, \mbox{ for } n\geq 1.
\end{eqnarray*}
\end{lem}

\begin{proof}
See \cite[Theorem 10.3.1]{Allouche} or \cite{Cobham}.
\end{proof}

\begin{lem}
Let ${\bf a}=(a_n)_{n\geq 0}$ be a primitive morphic sequence over a finite set of cardinality of $b\geq 2$.
Let $v$ be the width of $\sigma $ which generates the sequence ${\bf a}$.
Then we have
\begin{eqnarray*}
p({\bf a},n)\leq 2 v^{4b-2} b^3 n, \mbox{ for } n\geq 1.
\end{eqnarray*}
\end{lem}

\begin{proof}
See \cite[Theorem 10.4.12]{Allouche}.
\end{proof}

\begin{lem}
Let ${\bf a}$ be a Sturmian sequence.
Then we have
\begin{eqnarray*}
p({\bf a}, n)=n+1, \mbox{ for } n\geq 1.
\end{eqnarray*}
\end{lem}

\begin{proof}
See \cite[Theorem 10.5.8]{Allouche}.
\end{proof}

Consequently, by Theorem \ref{main1} and \ref{main2}, we obtain the upper bound of $w_2$ of automatic, primitive morphic, or Strumian continued fractions as follows:

\begin{thm}
Let $k\geq 2$ be an integer.
Let ${\bf a}=(a_n)_{n\geq 0}$ be a non-ultimately periodic and $k$-automatic sequence over $\lF _q[T]$ with $\deg a_n \geq 1$ for all $n\geq 0$.
Let $A$ be an upper bound of the sequence $(|a_n|)_{n\geq 0}$, $m$ be a cardinality of $k$-kernel of ${\bf a}$, and $d$ be a cardinality of the initial alphabet associated with ${\bf a}$.
Set $\xi :=[0,a_0,a_1,\ldots ]$.
Then we have 
\begin{eqnarray*}
w_2 (\xi )\leq 128 (2 k d^2+1)^3 k^m \left( \frac{\log A}{\log q} \right) ^4.
\end{eqnarray*}
\end{thm}

\begin{thm}
Let ${\bf a}=(a_n)_{n\geq 0}$ be a non-ultimately and primitive morphic sequence over $\lF _q [T]$ with $\deg a_n \geq 1$ for all $n\geq 0$, which is generated by a primitive morphism $\sigma $ over a finite set of cardinality $b\geq 2$.
Let $v$ be the width of $\sigma $.
Set $\xi :=[0,a_0,a_1,\ldots ]$.
Then we have
\begin{eqnarray*}
w_2 (\xi )\leq 128 (4 v^{4 b-2}b^3+1)^3 \Dio ({\bf a}) \left( \frac{\log A}{\log q} \right) ^4.
\end{eqnarray*}
\end{thm}

\begin{thm}
Let ${\bf a}=(a_n)_{n\geq 0}$ be a non-ultimately and Strumian sequence over $\lF _q [T]$ with $\deg a_n \geq 1$ for all $n\geq 0$.
Set $\xi :=[0,a_0,a_1,\ldots ]$.
Then we have
\begin{eqnarray*}
w_2 (\xi )\leq 16000\Dio ({\bf a}) \left( \frac{\log A}{\log q} \right) ^4
\end{eqnarray*}
if the slope of ${\bf a}$ has bounded partial quotients, and we have $w_2(\xi )=+\infty $ otherwise.
\end{thm}

\section{Liouville inequalities}\label{Liouville sec}

The following lemma is well-known and immediately seen.

\begin{lem}\label{height}
Let $P(X)$ be in $(\lF _q[T])[X]$.
Assume that $P(X)$ can be factorized as
\begin{eqnarray*}
P(X)=A\prod_{i=1}^{n} (X-\alpha _i),
\end{eqnarray*}
where $A\in \lF _q[T]$ and $\alpha _i \in \overline{\lF _q(T)}$ for $1\leq i\leq n$.
Then we have
\begin{align*}
H(P)=|A|\prod_{i=1}^{n} \max (1, |\alpha _i|).
\end{align*}
Furthermore, for $P(X), Q(X) \in (\lF _q[T])[X]$, we have
\begin{align*}
H(P Q)=H(P)H(Q).
\end{align*}
\end{lem}

The lemma below is an analogue of Theorem A.1 in \cite{Bugeaud1}.

\begin{prop}\label{Liouville inequ1}
Let $ P(X), Q(X) \in (\lF_q[T])[X]$ be non-constant polynomials of degree $m,n$, respectively.
Let $\alpha $ be a root of $P(X)$ of order $t$ and $\beta $ be a root of $Q(X)$ of order $u$.
Assume that $P(\beta )\neq 0$.
Then we have
\begin{eqnarray}\label{Liouville0}
|P(\beta )|\geq \max (1,|\beta |)^m H(P)^{-n/u+1} H(Q)^{-m/u}.
\end{eqnarray}
Furthermore, we have
\begin{gather}\label{Liouville1}
|\alpha -\beta |\geq \max (1, |\alpha |) \max (1, |\beta |) H(P)^{-n/t u} H(Q)^{-m/t u}.
\end{gather}
\end{prop}

\begin{proof}
Write $P(X)=A\prod_{i=1}^{r} (X-\alpha _i)^{t_i}$ and $Q(X)=B\prod_{i=1}^{s} (X-\beta _i)^{u_i}$, where $\alpha =\alpha _1, \beta =\beta _1, t=t_1 , u=u_1$, and $\alpha $'s (resp.\ $\beta $'s) are pairwise distinct.
Let $Q_1(X)=B_1\prod_{i=1}^{s_1} (X-\beta ^{(i)})^g$ be the minimal polynomial of $\beta $, where $\beta ^{(1)} =\beta , g=\insep \beta $, and $\beta ^{(i)}$'s are pairwise distinct.
Since $P$ and $Q_1$ do not have common roots, the resultant $\Res (P,Q_1)$ is non-zero and is in $\lF _q[T]$.
Therefore, by $H(Q_1)^{u/g} \leq H(Q)$ and $s_1 u\leq n$, we obtain
\begin{eqnarray*}
1 & \leq & |\Res (P,Q_1)| = |B_1|^m \prod_{i=1}^{s_1} |P(\beta ^{(i)})|^g \\
& \leq & |B_1|^m |P(\beta )|^g H(P)^{(s_1-1)g}\prod_{i=2}^{s_1} \max (1, |\beta ^{(i)}|)^{mg} \\
& = & |P(\beta )|^g H(P)^{(s_1-1)g} \left( \frac{H(Q_1)}{\max (1, |\beta |)^g} \right) ^m \\
& \leq & |P(\beta )|^g H(P)^{(n/u-1)g} H(Q)^{mg/u} \max (1, |\beta |)^{-mg}.
\end{eqnarray*}
As a result, we have (\ref{Liouville0}).
From Lemma \ref{height}, it follows that
\begin{eqnarray*}
|P(\beta )| & \leq & |\beta -\alpha |^t |A| \max (1, |\beta |)^{m-t}\prod_{i=2}^{r}\max (1, |\alpha _i|)^{t_i} \\
& = & |\beta -\alpha |^t H(P) \max (1, |\alpha |)^{-t} \max (1, |\beta |)^{m-t}.
\end{eqnarray*}
Hence, we have (\ref{Liouville1}) by (\ref{Liouville0}).
\end{proof}

The lemma below is an analogue of Theorem A.3 in \cite{Bugeaud1} and Lemma 2.3 in \cite{Pejkovic}.

\begin{lem}\label{Liouville inequ2}
Let $P(X) \in (\lF_q[T])[X]$ be an irreducible polynomial of degree $n \geq 2$.
For any distinct roots $\alpha ,\beta $ of $P(X)$, we have
\begin{eqnarray}\label{Liouville2}
|\alpha -\beta | \geq H(P)^{-n/f^2 +1/f},
\end{eqnarray}
where $f$ is inseparable degree of $P(X)$.
\end{lem}

\begin{proof}
We can write $P(X)=A\prod_{i=1}^{m} (X-\alpha _i)^f ,$ where $\alpha _1=\alpha , \alpha _2=\beta $ and $\alpha $'s are pairwise distinct.
Put $Q(X):=A\prod_{i=1}^{m} (X-\alpha _i ^f)$.
Since $Q(X)$ is separable, the discriminant $\Disc (Q)$ is non-zero and is in $\lF _q[T]$.
Therefore, we obtain
\begin{eqnarray*}
1 & \leq & |\Disc (Q)|
\leq |\alpha ^f-\beta ^f|^2 |A|^{2m-2} \prod_{\stackrel{1\leq i<j\leq m}{(i,j)\neq (1,2)}}^{} \max (1, |\alpha _i ^f|)^2\max (1, |\alpha _j ^f|)^2 \\
& \leq & |\alpha -\beta |^{2f} H(Q)^{2m-2} .
\end{eqnarray*}
Hence, we have (\ref{Liouville2}) by $H(P)=H(Q)$ and $n=mf$.
\end{proof}

The following proposition is an analogue of Corollary A.2 in \cite{Bugeaud1} and Lemma 2.5 in \cite{Pejkovic}, and is an extension of Theorem 1 in \cite{Mahler3}.

\begin{prop}\label{Liouville}
Let $\alpha ,\beta \in \overline{\lF _q(T)}$ be distinct algebraic numbers of degree $m, n$ and inseparable degree $f,g$, respectively.
Then we have
\begin{eqnarray}
|\alpha -\beta |\geq \max (1,|\alpha |) \max (1,|\beta |) H(\alpha )^{-n/f g} H(\beta )^{-m/f g}.
\end{eqnarray}
\end{prop}

\begin{proof}
From (\ref{Liouville1}) and (\ref{Liouville2}), the above inequality immediately holds.
\end{proof}

Let $\alpha \in \overline{\lF _q(T)}$ be a quadratic number.
Then we denote by $\alpha '$ the Galois conjugate of $\alpha $ which is different from $\alpha $ if $\insep \alpha =1$, and itself if $\insep \alpha =2$.
The lemma below is an analogue of Lemma 3.2 in \cite{Pejkovic}.

\begin{lem}\label{Galois conj}
Let $\alpha \in \overline{\lF _q(T)}$ be a quadratic number.
If $\alpha \neq \alpha '$, then we have
\begin{eqnarray}
H(\alpha )^{-1}\leq |\alpha -\alpha '|\leq H(\alpha ).
\end{eqnarray}
\end{lem}

\begin{proof}
Let $P_{\alpha }(X)= AX^2+BX+C$ be the minimal polynomial of $\alpha $.
Then we have
\begin{gather*}
|\alpha -\alpha '| = \frac{|B^2-4AC|^{1/2}}{|A|} \leq \max (|B|, |AC|^{1/2}) \leq H(\alpha )\\
\intertext{and}
|\alpha -\alpha '|\geq \frac{1}{|A|} \geq H(\alpha )^{-1}.
\end{gather*}
\end{proof}

We give a better estimate than Proposition \ref{Liouville} in some cases, which is an analogue of Lemma 7.1 in \cite{Bugeaud2} and Lemma 4 in \cite{Bugeaud3}.

\begin{prop}\label{Liouvilleinequ3}
Let $\alpha ,\beta \in \overline{\lF _q(T)}$ be quadratic numbers.
We denote by $P_{\alpha }(X)=A(X-\alpha )(X-\alpha '), P_{\beta }(X)=B(X-\beta )(X-\beta ')$ the minimal polynomials of $\alpha ,\beta $, respectively.
If $\alpha \neq \alpha '$ and $P_{\alpha }(X)\neq P_{\beta }(X)$, then we have
\begin{eqnarray}\label{Liouville3}
|\alpha -\beta |\geq \max (1, |\alpha -\alpha '|^{-1}) H(\alpha )^{-2} H(\beta )^{-2}.
\end{eqnarray}
\end{prop}

\begin{proof}
By Proposition \ref{Liouville}, we may assume that $|\alpha -\alpha '|<1$.
Since $P_{\alpha }(X)$ and $P_{\beta }(X)$ does not have common roots, we have
\begin{eqnarray*}
1 & \leq & |\Res (P_{\alpha },P_{\beta })| = |B|^2 |P_{\alpha }(\beta )||P_{\alpha }(\beta ')| \\
& \leq & |AB^2||\alpha -\beta ||\alpha '-\beta |H(\alpha ) \max (1,|\beta '|)^2 \\
& \leq & |\alpha -\beta ||\alpha '-\beta |H(\alpha )^2 H(\beta )^2.
\end{eqnarray*}
In the case of $|\alpha '-\beta |>|\alpha -\beta |$, we have $|\alpha -\alpha '|=|\alpha '-\beta |$.
Hence, we get (\ref{Liouville3}).
In other case, using Lemma \ref{Galois conj}, we obtain
\begin{eqnarray*}
|\alpha -\beta |^2 & \geq & |\alpha -\beta ||\alpha '-\beta | \geq H(\alpha )^{-2}H(\beta )^{-2}
\geq |\alpha -\alpha '|^{-2} H(\alpha )^{-4}H(\beta )^{-4}.
\end{eqnarray*}
Therefore, we have (\ref{Liouville3}).
\end{proof}

\section{Continued fractions}\label{Continued sec}

We collect fundamental properties of continued fractions for Laurent series over a finite field.
The lemma below is immediate by induction on $n$.

\begin{lem}\label{fund}
Consider a continued fraction $\xi =[a_0,a_1,a_2, \ldots ]\in \lF_q((T^{-1}))$.
Let \\ $(p_n/q_n)_{n\geq 0}$ be the convergent sequence of $\xi $.
Then the following hold: for any $n\geq 0$,
\begin{enumerate}
\item[(i)] $q_n p_{n-1} -p_n q_{n-1} =(-1)^n$,
\item[(ii)] $(p_n,q_n)=1$,
\item[(iii)] $|q_n|=|a_1||a_2|\cdots |a_n|$,
\item[(iv)] $\xi = \frac{\xi _{n+1}p_n+p_{n-1}}{\xi _{n+1}q_n +q_{n-1}}$, where $\xi =[a_0,\ldots ,a_n,\xi _{n+1}]$,
\item[(v)] $\left| \xi -p_n/q_n \right| =|q_n|^{-1}|q_{n+1}|^{-1}=|a_{n+1}|^{-1}|q_n|^{-2}$,
\item[(vi)] $q_n/q_{n-1}=[a_n,a_{n-1},\ldots a_1]$.
\end{enumerate}
\end{lem}

We recall an analogue of Lagrange's theorem for Laurent series over a finite field.

\begin{thm}\label{Lagrange}
Let $\xi $ be in $\lF _q((T^{-1}))$.
Then $\xi $ is quadratic if and only if its continued fraction expansion is ultimately periodic.
\end{thm}

\begin{proof}
See e.g. \cite[Theorem 3 and 4]{Chaichana}.
\end{proof}

The lemma below is immediate by Lemma \ref{fund}.

\begin{lem}\label{height upper}
Consider an ultimately periodic continued fraction
\begin{eqnarray*}
\xi =[0,a_1,\ldots ,a_r,\overline{a_{r+1},\ldots ,a_{r+s}}]\in \lF _q((T^{-1}))
\end{eqnarray*}
for $r\geq 0, s\geq 1.$
Let $(p_n/q_n)_{n\geq 0}$ be the convergent sequence of $\xi $.
Then $\xi $ is a root of the following equation:
\begin{eqnarray*}
(q_{r-1}q_{r+s}-q_r q_{r+s-1})X^2-(q_{r-1}p_{r+s}-q_r p_{r+s-1}+p_{r-1}q_{r+s}-p_r q_{r+s-1})X \\
+p_{r-1}p_{r+s}-p_r p_{r+s-1}=0,
\end{eqnarray*}
and we have $\ H(\xi )\leq |q_r q_{r+s}|$.
In particular, if $\xi =[0,\overline{a_1,\ldots ,a_s}]$, then $\xi $ is a root of the following equation:
\begin{eqnarray*}
q_{s-1}X^2-(p_{s-1}-q_s)X-p_s=0,
\end{eqnarray*}
and we have $H(\xi )\leq |q_s|$.
\end{lem}

\begin{lem}\label{lower bound}
Let $M\geq q$ be an integer and $\xi =[0, a_1,a_2,\ldots ], \zeta =[0,b_1,b_2,\ldots ] \in \lF _q((T^{-1}))$ be continued fractions with $|a_n|,|b_n|\leq M$ for all $n\geq 1$.
Assume that there exists an integer $n_0\geq 1$ such that $a_n=b_n$  for all $1\leq n\leq n_0$ and $a_{n_0+1}\neq  b_{n_0+1}$.
Then we have
\begin{eqnarray*}
|\xi -\zeta |\geq \frac{1}{M^2|q_{n_0}|^2}.
\end{eqnarray*}
\end{lem}

\begin{proof}
See \cite[Lemma 3]{Adamczewski3}.
\end{proof}

\begin{lem}\label{conj}
For $n\geq 0$, consider an ultimately periodic continued fraction \\
$\xi =[\overline{a_0,a_1,\ldots ,a_n}] \in \lF _q((T^{-1}))$ with $\deg a_0\geq 1$.
Then we have
\begin{eqnarray*}
-\frac{1}{\xi '}=[\overline{a_n,a_{n-1},\ldots ,a_0}].
\end{eqnarray*}
\end{lem}

\begin{proof}
See the proof of Lemma 2 in \cite{Dubois1}.
\end{proof}

The following lemma is an analogue of Lemma 6.1 in \cite{Bugeaud2}.

\begin{lem}\label{conj2}
For $r,s\geq 1$, consider an ultimately periodic continued fraction \\
$\xi =[0,a_1,\ldots ,a_r,\overline{a_{r+1},\ldots ,a_{r+s}}] \in \lF _q((T^{-1}))$ with $a_r\neq a_{r+s}$.
Let $(p_n/q_n)_{n\geq 0}$ be the convergent sequence of $\xi $.
Then we have
\begin{eqnarray}\label{conj3}
\frac{\min (|a_r|, |a_{r+s}| )}{|q_r|^2}\leq |\xi -\xi '|\leq \frac{|a_r a_{r+s}|}{|q_r|^2}.
\end{eqnarray}
\end{lem}

\begin{proof}
Put $\tau :=[\overline{a_{r+1},\ldots ,a_{r+s}}]$.
By Lemma \ref{conj}, we have $\tau '=-[0,\overline{a_{r+s},\ldots ,a_{r+1}}]$.
Since
\begin{eqnarray*}
\xi =\frac{p_r \tau +p_{r-1}}{q_r \tau +q_{r-1}},\quad \xi ' =\frac{p_r \tau '+p_{r-1}}{q_r \tau '+q_{r-1}},
\end{eqnarray*}
we obtain
\begin{eqnarray*}
|\xi -\xi '|=\frac{|\tau -\tau '|}{|q_r \tau +q_{r-1}||q_r \tau '+q_{r-1}|}
\end{eqnarray*}
by Lemma \ref{fund} (i).
We see $|\tau -\tau '|=|a_{r+1}|$ and $|q_r \tau +q_{r-1}|=|q_r||a_{r+1}|$.
It follows from Lemma \ref{fund} (vi) that
\begin{eqnarray*}
|q_r \tau '+q_{r-1}| & = & |q_r|\left| \tau '+\frac{q_{r-1}}{q_r}\right|
= \frac{|q_r||[\overline{a_{r+s},\ldots ,a_{r+1}}]-[a_r,\ldots ,a_1]|}{|[\overline{a_{r+s},\ldots ,a_{r+1}}]||[a_r,\ldots ,a_1]|} \\
& = & \frac{|q_r||a_{r+s}-a_r|}{|a_r a_{r+s}|}.
\end{eqnarray*}
Since $1\leq |a_{r+s}-a_r|\leq \max (|a_{r+s}|,|a_r|)$, we obtain (\ref{conj3}).
\end{proof}

The lemma below is an analogue of Lemma 6.3 in \cite{Bugeaud2}.

\begin{lem}\label{h and l}
Let $b,c,d\in \lF_q[T]$ be distinct polynomials of degree at least one, $n\geq 1$ be an integer, and $a_1,\ldots ,a_{n-1} \in \lF _q[T]$ be polynomials of degree at least one.
Put
\begin{eqnarray*}
\xi :=[0,a_1,\ldots ,a_{n-1},c,\overline{b}].
\end{eqnarray*}
Then $\xi $ is quadratic and
\begin{eqnarray*}
H(\xi )\asymp _{b,c} |q_n|^2,
\end{eqnarray*}
where $(p_k/q_k)_{k\geq 0}$ is the convergent sequence of $\xi $.
Let $m\geq 2$ be an integer.
Set
\begin{eqnarray*}
\zeta :=[0,a_1,\ldots ,a_{n-1},c,\overline{b,\ldots ,b,d}],
\end{eqnarray*}
where the length of period part of $\zeta $ is $m$. 
Then $\zeta $ is quadratic and
\begin{eqnarray*}
H(\zeta )\asymp _{b,c,d} |\tilde{q}_n \tilde{q}_{n+m}|,
\end{eqnarray*}
where $(\tilde{p}_k/\tilde{q}_k)_{k\geq 0}$ is the convergent sequence of $\zeta $.
\end{lem}

\begin{proof}
It follows from Theorem \ref{Lagrange} that $\xi $ and $\zeta $ are quadratic.
By Lemma \ref{height upper}, we have $H(\xi ) \ll _{b,c} |q_n|^2$ and $H(\zeta ) \ll _{b,c,d} |\tilde{q}_n \tilde{q}_{n+m}|$.
Let $P_\xi (X)=A(X-\xi )(X-\xi ')$ be the minimal polynomial of $\xi $.
Since $P_\xi (p_n/q_n)$ is non-zero, we obtain $|P_\xi (p_n/q_n)|\geq 1/|q_n|^2$.
From Lemma \ref{fund} (v) and \ref{conj2}, it follows that
\begin{eqnarray*}
\left| \xi -\frac{p_n}{q_n}\right| , \left| \xi '-\frac{p_n}{q_n}\right| \ll _{b,c} \frac{1}{|q_n|^2 }.
\end{eqnarray*}
Therefore, we obtain $|q_n|^2 \ll _{b,c} |A| \ll _{b,c} H(\xi )$.
We denote by $P_\zeta (X)$ the minimal polynomial of $\zeta $.
Since $P_\zeta $ and $P_\xi $ do not have a common root, we have
\begin{eqnarray*}
1\leq |\Res (P_\zeta ,P_\xi )| \leq H(\zeta )^2 H(\xi )^2 |\xi -\zeta ||\xi '-\zeta ||\xi -\zeta '||\xi '-\zeta '|.
\end{eqnarray*}
Note that $q_n=\tilde{q}_n$.
By Lemma \ref{conj}, we obtain
\begin{eqnarray*}
|\xi -\zeta |\ll _{b,c,d} |\tilde{q}_{n+m}|^{-2},\quad |\xi '-\zeta |, |\xi -\zeta '|, |\xi '-\zeta '|\ll _{b,c,d} |\tilde{q}_n|^{-2}.
\end{eqnarray*}
Therefore, it follows that $1 \ll _{b,c,d} H(\zeta )^2 H(\xi )^2 |\tilde{q}_n|^{-6} |\tilde{q}_{n+m}|^{-2}$.
Hence, we have the inequality $|\tilde{q}_n \tilde{q}_{n+m}|\ll _{b,c,d} H(\zeta )$.
\end{proof}

The next lemma is a well-known result.

\begin{lem}\label{irr ex}
Consider a continued fraction $\xi =[a_0,a_1,a_2,\ldots ]\in \lF _q((T^{-1}))$.
Let $(p_n/q_n)_{n\geq 0}$ be the convergent sequence of $\xi $.
Then we have
\begin{eqnarray*}
w_1 (\xi )= \limsup_{n\rightarrow \infty } \frac{\deg q_{n+1}}{\deg q_n}.
\end{eqnarray*}
\end{lem}

The lemma below is an analogue of Lemma 5.6 in \cite{Adamczewski3.5}.

\begin{lem}
Consider a continued fraction $\xi =[a_0,a_1,a_2,\ldots ]\in \lF _q((T^{-1}))$.
Let $(p_n/q_n)_{n\geq 0}$ be the convergent sequence of $\xi $.
If the sequence $(|q_n|^{1/n})_{n\geq 1}$ is bounded, then $\xi $ is not a $U_1$-number.
\end{lem}

\begin{proof}
By the assumption, there exists an integer $A$ such that $q^n \leq |q_n|\leq A^n$ for all $n\geq 1$.
Thus, for all $n\geq 1$, we have
\begin{eqnarray*}
\frac{\deg q_{n+1}}{\deg q_n} \leq \left( 1+\frac{1}{n} \right) \frac{\log A}{\log q}.
\end{eqnarray*}
By Lemma \ref{irr ex}, we obtain $w_1 (\xi )\leq \log A/\log q$.
\end{proof}

\section{Properties of $w_n$ and $w_n ^{*}$}\label{Properties sec}

\begin{thm}\label{Mahler lower}
Let $n\geq 1$ be an integer and $\xi \in \lF _q((T^{-1}))$ be not algebraic of degree at most $n$.
Then we have
\begin{align*}
w_n (\xi ) \geq n,\quad w_n ^{*}(\xi ) & \geq \frac{n+1}{2}.
\end{align*}
Furthermore, if $n=2$, then $w_2 ^{*}(\xi )\geq 2$.
\end{thm}

\begin{proof}
The former estimate follows from an analogue of Minkowski's theorem for Laurent series over a finite field \cite{Mahler2} and the later estimates are Satz.1 and Satz.2 of \cite{Guntermann}.
\end{proof}

We give an immediate consequence of Proposition \ref{Liouville inequ1} and \ref{Liouville}.

\begin{thm}\label{alg}
Let $n\geq 1$ be an integer and $\xi \in \lF _q((T^{-1}))$ be an algebraic number of degree $d$.
Then we have
\begin{align*}
w_n (\xi ), w_n ^{*}(\xi ) \leq d-1.
\end{align*}
\end{thm}

Note that if $\xi \in \lF _q((T^{-1}))$ be an algebraic number, then $\insep \xi =1$.

We next show that we can replace the definition of $w_n$ by a weak definition of $w_n$.
Let $n\geq 1$ be an integer and $\xi $ be in $\lF _q((T^{-1}))$.
We define a Diophantine exponent $\tilde{w}_n$ at $\xi $ by the supremum of a real number $w$ for which there exist infinitely many $P(X) \in (\lF _q[T])[X]_{\min }$ of degree at most $n$ such that
\begin{align*}
0<|P(\xi )|\leq H(P)^{-w}.
\end{align*}

The lemma below is a slight improvement of a result in \cite[Section 3.4]{Sprindzuk1}.

\begin{lem}\label{ok}
Let $n\geq 1$ be an integer and $\xi $ be in $\lF _q((T^{-1}))$.
Then we have
\begin{eqnarray*}
w_n (\xi )=\tilde{w}_n (\xi ).
\end{eqnarray*}
\end{lem}

\begin{proof}
It is immediate that $\tilde{w}_n (\xi )\leq w_n(\xi )$ and $\tilde{w}_n (\xi )\geq 0$.
Therefore, we may assume that $w_n(\xi )>0$ and $\tilde{w}_n (\xi )$ is finite.
For $0<w<w_n(\xi )$, there exist infinitely many $P(X)\in (\lF _q[T])[X]$ of degree at most $n$ such that
\begin{eqnarray}
0<|P(\xi )|\leq H(P)^{-w}.
\end{eqnarray}
We can write $P(X)=A\prod _{i=1}^{k}P_i(X),$ where $A\in \lF _q[T]$ and $P_i(X) \in (\lF _q[T])[X]_{\min }$ for $1\leq i\leq k$.
By the definition, for $\tilde{w}>\tilde{w}_n(\xi )$, there exists a positive number $C$ such that for all $Q(X)\in (\lF _q[T])[X]_{\min }$ of degree at most $n$, 
\begin{eqnarray*}
|Q(\xi )|\geq CH(Q)^{-\tilde{w}}. 
\end{eqnarray*}
Therefore, by Lemma \ref{height}, we obtain
\begin{eqnarray*}
|P(\xi )|\geq \min (1,C^n)H(P)^{-\tilde{w}},
\end{eqnarray*}
which implies $\min (1,C^n)H(P)^w\leq H(P)^{\tilde{w}}$.
Since there exist infinitely many such polynomials $P(X)$, we have $w\leq \tilde{w}$.
This completes the proof.
\end{proof}

The lemma below is an analogue of Lemma A.8 in \cite{Bugeaud1} and Lemma 2.4 in \cite{Pejkovic}.

\begin{lem}\label{insep}
Let $P(X) \in (\lF _q[T])[X]$ be a non-constant irreducible polynomial of degree $n$ and of inseparable degree $f$.
Let $\xi $ be in $\lF _q((T^{-1}))$ and $\alpha $ be a root of $P(X)$ such that $|\xi -\alpha |$ is minimal.
If $n\geq 2f$, then we have
\begin{eqnarray}\label{comp}
|\xi -\alpha |\leq |P(\xi )|^{1/f}H(P)^{n/f^2-2/f}.
\end{eqnarray}
\end{lem}

\begin{proof}
We may assume that $\xi $ and $\alpha $ are distinct.
We first consider the case of $f=1$ .
Write $P(X)=A\prod _{i=1}^{n} (X-\alpha _i)$, where $\alpha =\alpha _1$ and $|\xi -\alpha _1|\leq |\xi-\alpha _2|\leq \ldots \leq |\xi -\alpha _n|$.
Put $Q(X):=A\prod _{i=2}^{n} (X-\alpha _i)$ and $\Delta :=\prod _{i=2}^{n}|\alpha -\alpha _i|$.
Then we have $|\Disc (P)|^{1/2}=\Delta |A||\Disc (Q)|^{1/2}$.
By the definition of discriminant, we obtain
\begin{eqnarray*}
|\Disc (Q)|^{1/2} & = & |A|^{n-2}|\det (\alpha _i ^j)_{2\leq i\leq n, 0\leq j\leq n-2}| \\
& \leq & |A|^{n-2}\prod _{i=2}^{n} \max (1,|\alpha _i|)^{n-2} \\
& = & H(P)^{n-2} \max (1,|\alpha |)^{-n+2}.
\end{eqnarray*}
Since the polynomial $P$ is separable, we get
\begin{eqnarray*}
1 & \leq & |\Disc (P)|^{1/2}
\leq H(P)^{n-2}\max (1,|\alpha |)^{-n+2}|A|\prod_{j=2}^{n}|\xi -\alpha _j|\\
& = & H(P)^{n-2}\max (1,|\alpha |)^{-n+2}|\xi -\alpha |^{-1}|P(\xi )|.
\end{eqnarray*}
Therefore, we have (\ref{comp}).

We next consider the case of $f>1$.
We can write $P(X)=R(X^f)$, where a separable polynomial $R(X) \in (\lF _q[T])[X]$.
Thus, in the same way, it follows that
\begin{eqnarray*}
|\xi ^f -\alpha ^f|\leq |R(\xi ^f)|H(R)^{n/f-2}.
\end{eqnarray*}
Since $H(P)=H(R)$ and $f$ is a power of $p$, we have (\ref{comp}).
\end{proof}

\begin{lem}\label{good}
Let $\xi $ be in $\lF _q((T^{-1}))$ and $n\geq 1$ be an integer.
Then we have
\begin{eqnarray*}
w_1(\xi )=w_1(\xi ^{p^n}).
\end{eqnarray*}
\end{lem}

\begin{proof}
By Theorem \ref{alg}, we may assume that $\xi $ is not in $\lF _q(T)$.
Therefore, we can write $\xi =[a_0,a_1,\ldots ]$.
Then we have $\xi ^{p^n}=[a_0 ^{p^n},a_1 ^{p^n},\ldots ]$ by the Frobenius endmorphism.
Hence, it follows from Lemma \ref{fund} (iii) and \ref{irr ex} that
\begin{eqnarray*}
w_1(\xi ^{p^n})=\limsup _{k\rightarrow \infty }\frac{\sum _{i=1}^{k+1}\deg a_{i} ^{p^n}}{\sum _{i=1}^{k}\deg a_i ^{p^n}}=\limsup _{k\rightarrow \infty }\frac{p^n \deg q_{k+1}}{p^n \deg q_k}=w_1(\xi ),
\end{eqnarray*}
where $(p_k/q_k)_{k\geq 0}$ is the convergent sequence of $\xi $.
\end{proof}

\begin{prop}\label{main7}
Let $n\geq 1$ be an integer and $\xi $ be in $\lF _q((T^{-1}))$.
Let $k\geq 0$ be an integer such that $p^k\leq n<p^{k+1}$.
Then we have
\begin{eqnarray}
\frac{w_n(\xi )}{p^k}-n+\frac{2}{p^k}-1\leq w_n ^{*}(\xi )\leq w_n(\xi ).
\end{eqnarray} 
Furthermore, if $1\leq n<2p$, then we have
\begin{eqnarray}\label{dif}
w_n(\xi )-n+1\leq w_n ^{*}(\xi )\leq w_n(\xi ).
\end{eqnarray}
\end{prop}

\begin{rem}
We are able to define analogues of Diophantine exponents $w_n$ and $w_n ^{*}$ for real numbers and $p$-adic numbers (see \cite[Section 3.1, 3.3, and 9.3]{Bugeaud1} for the definition of $w_n$ and $w_n ^{*}$).
It is known that for all $n$, analogues of (\ref{dif}) for real numbers and $p$-adic numbers hold (see \cite{Wirsing,Morrison}).
However, in our framework, we are not able to prove (\ref{dif}) for all $n$.
The main difficulty is the existence of inseparable irreducible polynomials in $(\lF _q[T])[X]$.
Therefore, it seems that Proposition $\ref{main7}$ describes the difference between approximation properties of characteristic zero and that of positive characteristic.
On the other hand, when $n$ is sufficiently small, we prove (\ref{dif}) using continued fraction theory and the Frobenius endomorphism. 
\end{rem}

\begin{proof}
It is immediate that $w_n(\xi ), w_n ^{*}(\xi )\geq 0$.
We first show that $w_n ^{*}(\xi )\leq w_n(\xi )$.
We may assume that $w_n ^{*}(\xi )>0$.
For $0<w^{*}<w_n ^{*}(\xi )$, there exist infinitely many $\alpha \in \overline{\lF _q(T)}$ of degree at most $n$ such that
\begin{eqnarray*}
0<|\xi -\alpha |\leq H(\alpha )^{-w^{*}-1}.
\end{eqnarray*}
Let $P_{\alpha }(X)=\sum_{i=0}^{d}a_i X^i$ be the minimal polynomial of $\alpha $, where $d=\deg \alpha $.
Put 
\begin{eqnarray*}
Q_{\alpha }(X):=a_d X^{d-1}+(a_d \alpha +a_{d-1})X^{d-2}+(a_d \alpha ^2+a_{d-1} \alpha +a_{d-2})X^{d-3}\\
+\cdots +(a_d \alpha ^{d-1}+a_{d-1} \alpha ^{d-2}+\cdots +a_1).
\end{eqnarray*}
Then we have $P_{\alpha }(X)=(X-\alpha )Q_{\alpha }(X)$.
Since $\max (1,|\alpha |)=\max (1,|\xi |)$, we obtain $|Q_{\alpha }(\xi )|\leq H(P_{\alpha })\max (1,|\xi |)^{2n}$.
Hence, it follows that
\begin{eqnarray*}
|P_{\alpha }(\xi )|\leq H(\alpha )^{-w^*} \max (1,|\xi |)^{2n},
\end{eqnarray*}
which gives $w^{*}\leq w_n(\xi)$.
Consequently, we have $w_n ^{*}(\xi )\leq w_n(\xi )$.

Our next claim is that $w_n(\xi )/p^k-n+2/p^k-1\leq w_n ^{*}(\xi )$.
We may assume that $w_n(\xi )>0$.
For $0<w<w_n(\xi )$, there exist infinitely many $P(X) \in (\lF _q[T])[X]_{\min }$ of degree at most $n$ such that
\begin{eqnarray*}
0<|P(\xi )|\leq H(P)^{-w}
\end{eqnarray*}
by Lemma \ref{ok}.
Let $m$ denote the degree of $P$ and $f$ denote the inseparable degree of $P$.
We first consider the case of $m\geq 2f$.
By Lemma \ref{insep}, there exists $\alpha $ of root of $P$ such that
\begin{eqnarray*}
|\xi -\alpha |\leq H(\alpha )^{-w/f+m/f^2-2/f}\leq H(\alpha )^{-w/p^k+n-2/p^k}.
\end{eqnarray*}
Now, assume that $m<2f$.
Then we have $m=f$ by $f|m$.
Therefore, we can write $P(X)=A(X^m-\alpha ^m)$, where $A\in \lF_q[T]$ and $\alpha \in \overline{\lF_q(T)}$.
Thus, we get $|\xi -\alpha |<|A|^{-1/f}H(\alpha )^{-w/f}$.
Since $\max (1,|\xi |)=\max (1,|\alpha |)$, we have
\begin{eqnarray*}
|\xi -\alpha |\leq \max (1,|\xi |)H(\alpha )^{-w/f-1/f}
\leq \max (1,|\xi |)H(\alpha )^{-w/p^k+n-2/p^k}
\end{eqnarray*}
by Lemma \ref{height}.
This is our claim.

Finally, we assume $1\leq n<2p$ and show (\ref{dif}).
Let $0<w<w_n(\xi )$.
If there exist infinitely many separable polynomials $P(X) \in (\lF _q[T])[X]_{\min }$ of degree at most $n$ such that
\begin{eqnarray*}
0<|P(\xi )|\leq H(P)^{-w},
\end{eqnarray*}
then we have $w-n+1\leq w_n ^{*}(\xi )$ as in the same line of the above proof.
Therefore, we may assume that there exist infinitely many inseparable polynomials $P(X) \in (\lF _q[T])[X]_{\min }$ of degree at most $n$ such that
\begin{eqnarray*}
0<|P(\xi )|\leq H(P)^{-w}.
\end{eqnarray*}
Then we can write such polynomials $P(X)=AX^p+B,$ where $A,B \in \lF _q[T]$.
By Lemma \ref{good} and the definition of $w_n$, we have $w\leq w_1(\xi )$.
Therefore, we obtain $w-n+1\leq w_n ^{*}(\xi )$ by $w_1(\xi )=w_1 ^{*}(\xi )$.
Hence, we have (\ref{dif}).
\end{proof}

It follows from Proposition \ref{main7} that for an integer $n\geq 1$ and $\xi \in \lF _q((T^{-1}))$
\begin{itemize}
 \item $w_n(\xi )$ is finite if and only if $w_n ^{*}(\xi )$ is finite,
 \item if $w(\xi )$ is finite, then $w^{*}(\xi )$ is finite.
\end{itemize}
Consequently, we obtain
\begin{itemize}
 \item $\xi $ is a $U_n$-number if and only if it is a $U_n ^{*}$-number,
 \item if $\xi $ is an $S$-number, then it is an $S^{*}$-number.
\end{itemize}

We address the following questions in the last of this section.

\begin{prob}
Does (\ref{dif}) hold for all $n\geq 1$ and $\xi \in \lF _q((T^{-1}))$?
\end{prob}

\begin{prob}\label{prob1}
Does Mahler's classification coincide Koksma's classification?
\end{prob}

Note that an analogue of Problem \ref{prob1} for real numbers and $p$-adic numbers holds.
The detail is found in \cite[Section 3.4 and 9.3]{Bugeaud1}, \cite[Chapter 6]{Pejkovic}, and \cite{Schmidt1}.

\section{Applications of Liouville inequalities}\label{Applications sec}

The following proposition is an analogue of  Lemma 7.3 in \cite{Bugeaud2} and Lemma 5 in \cite{Bugeaud3}.

\begin{prop}\label{Mahler and Koksma}
Let $\xi $ be in $\lF_q((T^{-1}))$ and $c_1,c_2,c_3,c_4,\theta ,\rho ,\delta $ be positive numbers.
Let $\varepsilon $ be a non-negative number.
Assume that there exists a sequence of distinct terms $(\alpha _j)_{j\geq 1}$ with $\alpha _j \in \overline{\lF _q(T)}$ is quadratic for $j\geq 1$ such that for all $j\geq 1$
\begin{gather*}
\frac{c_1}{H(\alpha _j)^{1+\rho }} \leq |\xi -\alpha _j|\leq \frac{c_2}{H(\alpha _j)^{1+\delta }},\\
H(\alpha _j)\leq  H(\alpha _{j+1})\leq c_3 H(\alpha _j)^\theta ,\\
0<|\alpha _j-\alpha _j '|\leq \frac{c_4}{H(\alpha _j)^\varepsilon }.
\end{gather*}
If $(\rho -1)(\delta -1+\varepsilon )\geq 2\theta (2-\varepsilon )$, then we have $\delta \leq w_2 ^{*}(\xi )\leq \rho $.
Furthermore, if $(\delta -1)(\delta -1+\varepsilon )\geq 2\theta (2-\varepsilon )$, then we have
\begin{eqnarray*}
\delta \leq w_2 ^{*}(\xi )\leq \rho ,\quad w_2(\xi )\geq w_2 ^{*}(\xi )+\varepsilon .
\end{eqnarray*}
Finally, assume that there exists a non-negative number $\chi $ such that
\begin{eqnarray*}
\limsup_{j\rightarrow \infty} \frac{-\log |\alpha _j-\alpha _j '|}{\log H(\alpha _j)}\leq \chi .
\end{eqnarray*}
If $(\delta -2+\chi )(\delta -1+\varepsilon )\geq 2\theta (2-\varepsilon )$ when $p\neq 2$ and $(\delta -4+\chi )(\delta -1+\varepsilon )\geq 4\theta (2-\varepsilon )$ when $p=2$, then we have
\begin{eqnarray*}
\delta \leq w_2 ^{*}(\xi )\leq \rho ,\quad \varepsilon \leq  w_2(\xi )-w_2 ^{*}(\xi ) \leq \chi .
\end{eqnarray*}
\end{prop}

\begin{proof}
Assume that $(\rho -1)(\delta -1+\varepsilon )\geq 2\theta (2-\varepsilon )$.
By the assumption, we have $\theta \geq 1, \rho >1$, and $\delta +\varepsilon >1$.
Let $\alpha \in \overline{\lF _q(T)}$ be an algebraic number of degree at most two with sufficiently large height and $\alpha \notin \{\alpha _j\mid j\geq 1 \} $.
We define an integer $j_0\geq 1$ by $H(\alpha _{j_0})\leq c_3 \{ (c_2 c_4)^{\frac{1}{2}}H(\alpha )\} ^{\frac{2\theta }{\delta +\varepsilon -1}}<H(\alpha _{j_0 +1})$.
Then, by the assumption, we have
\begin{align*}
H(\alpha )<c_3 ^{-\frac{\delta +\varepsilon -1}{2\theta }}(c_2 c_4)^{-\frac{1}{2}} H(\alpha _{j_0+1})^{\frac{\delta +\varepsilon -1}{2\theta }}
\leq (c_2 c_4)^{-\frac{1}{2}} H(\alpha _{j_0})^{\frac{\delta +\varepsilon -1}{2}}.
\end{align*}
Hence, it follows from Proposition \ref{Liouville}, \ref{Liouvilleinequ3}, and Lemma \ref{Galois conj} that
\begin{eqnarray*}
|\alpha -\alpha _{j_0}| & \geq & |\alpha _{j_0}-\alpha _{j_0} '|^{-1}H(\alpha _{j_0})^{-2}H(\alpha )^{-2}\\
& > & c_2 H(\alpha _{j_0})^{-\delta -1} \geq |\xi -\alpha _{j_0}|.
\end{eqnarray*}
Therefore, we obtain
\begin{eqnarray}
|\xi -\alpha | & = & |\alpha -\alpha _{j_0}|>c_4^{-1}H(\alpha _{j_0})^{-2+\varepsilon }H(\alpha )^{-2} \nonumber \\
\label{alg lower} & \geq & c_2^{-\frac{\theta (2-\varepsilon )}{\delta +\varepsilon -1}}c_3 ^{-2+\varepsilon }c_4 ^{-1-\frac{2\theta (2-\varepsilon )}{\delta +\varepsilon -1}} H(\alpha )^{-2-\frac{2\theta (2-\varepsilon )}{\delta +\varepsilon -1}}.
\end{eqnarray}
By the assumption, we have 
\begin{align*}
\delta \leq w_2 ^{*}(\xi )\leq \max \left( \rho ,1+\frac{2\theta (2-\varepsilon )}{\delta +\varepsilon -1} \right) =\rho .
\end{align*}

We next assume that $(\delta -1)(\delta -1+\varepsilon )\geq 2\theta (2-\varepsilon )$.
By (\ref{alg lower}), it follows that
\begin{align*}
|\xi -\alpha | \geq c_2^{-\frac{\theta (2-\varepsilon )}{\delta +\varepsilon -1}}c_3 ^{-2+\varepsilon }c_4 ^{-1-\frac{2\theta (2-\varepsilon )}{\delta +\varepsilon -1}} H(\alpha )^{-\delta -1}.
\end{align*}
Therefore,  the sequence $(\alpha _j)_{j\geq 1}$ is the best algebraic approximation to $\xi $ of degree at most two, that is,
\begin{align*}
w_2 ^{*}(\xi )= \limsup _{j\rightarrow \infty } \frac{-\log |\xi -\alpha _j|}{\log H(\alpha _j)} -1.
\end{align*}
We denote by $P_j (X)=A_j (X-\alpha _j)(X-\alpha _j ')$ the minimal polynomial of $\alpha _j$.
By the assumption, we have
\begin{align*}
|P_j (\xi )|\leq \max (c_2, c_4)H(\alpha _j)^{-\varepsilon +1} |\xi -\alpha _j|.
\end{align*}
Therefore, we obtain
\begin{eqnarray*}
w_2 ^{*}(\xi )+\varepsilon \leq \limsup _{n\rightarrow \infty }\frac{-\log |P_j (\xi )|}{H(P_j)} \leq w_2(\xi ).
\end{eqnarray*}

Finally, assume that
\begin{eqnarray*}
\limsup_{j\rightarrow \infty} \frac{\log |\alpha _j-\alpha _j '|}{\log H(\alpha _j)}=-\chi .
\end{eqnarray*}
We also assume that $(\delta -2+\chi )(\delta -1+\varepsilon )\geq 2\theta (2-\varepsilon )$ when $p\neq 2$ and $(\delta -4+\chi )(\delta -1+\varepsilon )\geq 4\theta (2-\varepsilon )$ when $p=2$.
Since $|\xi -\alpha _j|\leq c_2$ and $|\alpha _j-\alpha _j '|\leq c_4$, we have 
\begin{eqnarray*}
\max (1,|\alpha _j|), \max (1,|\alpha _j '|)\leq \max (1,c_2, c_4, |\xi |),
\end{eqnarray*}
which implies $H(P_j)\leq |A_j|\max (1,c_2, c_4, |\xi |)^2$.
By the assumption, we get $|\alpha _j-\alpha _j '|<|\xi -\alpha _j '|$ for sufficiently large $j$.
Therefore, we obtain for sufficiently large $j$
\begin{align*}
|P_j(\xi )|\geq \max (1,c_2, c_4, |\xi |)^{-2} H(P_j)|\xi -\alpha _j||\alpha _j -\alpha _j '|.
\end{align*}
Taking a logarithm and a limit superior, we have
\begin{align*}
\limsup _{j\rightarrow \infty } \frac{-\log |P_j(\xi )|}{\log H(P_j)} \leq w_2 ^{*}(\xi )+\chi .
\end{align*}

Let $P(X) \in (\lF_q[T])[X]_{\min }$ be a polynomial of degree at most two with sufficiently large height such that $P(\alpha _j)\neq 0$ for all $j\geq 1$.
When $\deg _X P=1$, we can write $P(X)=A(X-\alpha )$ and have
\begin{align*}
|P(\xi )|=|A||\xi -\alpha |\geq  c_2^{-\frac{\theta (2-\varepsilon )}{\delta +\varepsilon -1}}c_3 ^{-2+\varepsilon }c_4 ^{-1-\frac{2\theta (2-\varepsilon )}{\delta +\varepsilon -1}} H(P)^{-\delta -\chi }
\end{align*}
by (\ref{alg lower}).
When $\deg _X P=2$, we can write $P(X)=A(X-\alpha )(X-\alpha ')$ and assume that $|\xi -\alpha |\leq |\xi -\alpha '|$.
we first consider the case of $p\neq 2$.
Then we obtain
\begin{eqnarray}
|P(\xi )| & \geq & |A(\alpha -\alpha ')||\xi -\alpha | =|\Disc (P)|^{1/2}|\xi -\alpha | \nonumber \\
\label{before} & \geq & c_2^{-\frac{\theta (2-\varepsilon )}{\delta +\varepsilon -1}}c_3 ^{-2+\varepsilon }c_4 ^{-1-\frac{2\theta (2-\varepsilon )}{\delta +\varepsilon -1}} H(P)^{-\delta -\chi }
\end{eqnarray}
by (\ref{alg lower}).
We next consider the case of $p=2$.
If $\alpha \neq \alpha '$, then we have (\ref{before}) and if otherwise, then we have 
\begin{eqnarray*}
|P(\xi )| & \geq & |A||\xi -\alpha |^2
\geq c_2^{-\frac{2\theta (2-\varepsilon )}{\delta +\varepsilon -1}}c_3 ^{-4+2\varepsilon }c_4 ^{-2-\frac{4\theta (2-\varepsilon )}{\delta +\varepsilon -1}} H(P)^{-4-\frac{4\theta (2-\varepsilon )}{\delta +\varepsilon -1}} \\
& \geq & c_2^{-\frac{2\theta (2-\varepsilon )}{\delta +\varepsilon -1}}c_3 ^{-4+2\varepsilon }c_4 ^{-2-\frac{4\theta (2-\varepsilon )}{\delta +\varepsilon -1}} H(P)^{-\delta -\chi }
\end{eqnarray*}
by (\ref{alg lower}).
Thus, it follows that $w_2(\xi )\leq w_2 ^{*}(\xi )+\chi $ by Lemma \ref{ok}.
\end{proof}

The following proposition is an analogue of Lemma 7.2 in \cite{Bugeaud2}.

\begin{prop}\label{Mahler and Koksma2}
Let $\xi $ be in $\lF_q((T^{-1}))$.
Let $c_0,c_1,c_2,c_3, \theta ,\rho ,\delta $ be positive numbers and $(\beta _j)_{j\geq 1}$ be a sequence of positive integers such that $\beta _j<\beta _{j+1}\leq c_0 \beta _j ^\theta $ for all $j\geq 1$.
Assume that there exists a sequence of distinct terms $(\alpha _j)_{j\geq 1}$ with $\alpha _j \in \overline{\lF _q(T)}$ is quadratic for $j\geq 1$ such that for all $j\geq 1$
\begin{gather*}
\frac{c_1}{\beta _j ^{2+\rho }} \leq |\xi -\alpha _j|\leq \frac{c_2 \max (1, |\alpha _j -\alpha _j '|^{-1})}{\beta _j ^{2+\delta }},\\
H(\alpha _j)\leq c_3 \beta _j, \quad \alpha _j \neq \alpha _j '.
\end{gather*}
Then we have
\begin{eqnarray*}
1+\delta \leq w_2 ^{*} (\xi )\leq (2+\rho )\frac{2\theta }{\delta }-1.
\end{eqnarray*}
\end{prop}

\begin{proof}
Let $\alpha \in \overline{\lF _q(T)}$ be an algebraic number of degree at most two with sufficiently large height.
We define an integer $j_0\geq 1$ by $\beta _{j_0} \leq c_0 c_2^{\frac{\theta }{\delta }}(c_3 H(\alpha ))^{\frac{2\theta }{\delta }} <\beta _{j_0+1}$.
We first consider the case of $\alpha =\alpha _{j_0}$.
By the assumption, we have
\begin{eqnarray*}
|\xi -\alpha |\geq c_1 \beta _{j_0} ^{-2-\rho } \geq c_0 ^{-2-\rho }c_1 c_2 ^{-(2+\rho )\frac{\theta }{\delta }}c_3 ^{-(2+\rho )\frac{2\theta }{\delta }}H(\alpha )^{-(2+\rho )\frac{2\theta }{\delta }}.
\end{eqnarray*}

We next consider the other case.
Then, by the assumption, we have 
\begin{eqnarray*}
H(\alpha ) < c_2 ^{-\frac{1}{2}} c_3 ^{-1}(c_0 ^{-1} \beta _{j_0+1})^{\frac{\delta }{2\theta }} \leq c_2 ^{-\frac{1}{2}} c_3 ^{-1}\beta _{j_0} ^{\frac{\delta }{2}}.
\end{eqnarray*}
Hence, it follows from Proposition \ref{Liouville}, \ref{Liouvilleinequ3}, and Lemma \ref{Galois conj} that
\begin{eqnarray*}
|\alpha -\alpha _{j_0}| & \geq & \max (1,|\alpha _{j_0}-\alpha _{j_0} '|^{-1})H(\alpha _{j_0})^{-2}H(\alpha )^{-2} \\
& > &  c_2 \max (1,|\alpha _{j_0}-\alpha _{j_0} '|^{-1})\beta _{j_0} ^{-2-\delta }\geq |\xi -\alpha _{j_0}|.
\end{eqnarray*}
Therefore, we obtain
\begin{eqnarray*}
|\xi -\alpha | & = & |\alpha -\alpha _{j_0}| \geq \max (1,|\alpha _{j_0}-\alpha _{j_0} '|^{-1})H(\alpha _{j_0})^{-2}H(\alpha )^{-2}\\
& \geq & c_3 ^{-2}\beta _{j_0} ^{-2} H(\alpha )^{-2} \geq c_0 ^{-2}c_2 ^{-\frac{2\theta }{\delta }}c_3 ^{-2-\frac{4\theta }{\delta }}H(\alpha )^{-2-\frac{4\theta }{\delta }}.
\end{eqnarray*}
By the assumption, we have
\begin{eqnarray*}
1+\delta \leq w_2 ^{*}(\xi )\leq \max \left( 1+\frac{4\theta }{\delta } , (2+\rho )\frac{2\theta }{\delta }-1\right) =(2+\rho )\frac{2\theta }{\delta }-1.
\end{eqnarray*}
\end{proof}

\section{Combinational lemma}\label{Combinational sec}

The lemma below is a slight improvement of \cite[Lemma 9.1]{Bugeaud2}.

\begin{lem}\label{combi}
Let ${\bf a}=(a_n)_{n \geq 1}$ be a sequence on a finite set $\mathcal{A}$.
Assume that there exist integers $\kappa \geq 2$ and $n_0\geq 1$ such that for all $n\geq n_0$,
\begin{eqnarray*}
p({\bf a},n)\leq \kappa n.
\end{eqnarray*}
Then, for each $n\geq n_0$, there exist finite words $U_n, V_n$ and a positive rational number $w_n$ such that the following hold:
\begin{enumerate}
\item[(i)] $U_n V_n ^{w_n}$ is a prefix of ${\bf a}$,
\item[(ii)] $|U_n|\leq 2\kappa |V_n|$,
\item[(iii)] $n/2 \leq |V_n|\leq \kappa n$,
\item[(iv)] if $U_n$ is not an empty word, then the last letters of $U_n$ and $V_n$ are different,
\item[(v)] $|U_n V_n ^{w_n}|/|U_n V_n|\geq 1+1/(4\kappa +2)$,
\item[(vi)] $|U_n V_n|\leq (\kappa +1)n-1$,
\item[(vii)] $|U_n ^2 V_n|\leq (2\kappa +1)n-2$.
\end{enumerate}
\end{lem}

\begin{proof}
For $n\geq 1$, we denote by $A(n)$ the prefix of ${\bf a}$ of length $n$.
By Pigeonhole principle, for each $n\geq n_0$, there exists a finite word $W_n$ of length $n$ such that the word appears to $A((\kappa +1)n)$ at least twice.
Therefore, for each $n\geq n_0$, there exist finite words $B_n, D_n, E_n \in \mathcal{A}^{*}$ and $C_n \in \mathcal{A}^{+}$ such that
\begin{eqnarray*}
A((\kappa +1)n)=B_n W_n D_n E_n =B_n C_n W_n E_n.
\end{eqnarray*}
We can take these words in such way that if $B_n$ is not empty, then the last letter of $B_n$ is different from that of $C_n$.
Firstly, we consider the case of $|C_n|\geq |W_n|$.
Then, there exists $F_n \in \mathcal{A}^{*}$ such that
\begin{eqnarray*}
A((\kappa +1)n)=B_n W_n F_n W_n E_n.
\end{eqnarray*}
Put $U_n :=B_n, V_n:=W_n F_n$, and $w_n:=|W_n F_n W_n|/|W_n F_n|$.
Since $U_n V_n ^{w_n}=B_n W_n F_n W_n$, the word $U_n V_n ^{w_n}$ is a prefix of ${\bf a}$.
It is obvious that $|U_n|\leq (\kappa -1)|V_n|$ and $n\leq |V_n|\leq \kappa n$.
By the definition, we have (iv) and (vi).
Furthermore, we see that
\begin{gather*}
\frac{|U_n V_n ^{w_n}|}{|U_n V_n|}=1+\frac{n}{|U_n V_n|}\geq 1+\frac{1}{\kappa },\\
|U_n ^2 V_n|\leq |U_n V_n|+|U_n|\leq \kappa n+(\kappa -1)n=(2\kappa -1)n.
\end{gather*}
We next consider the case of $|C_n|< |W_n|$.
Since the two occurrences of $W_n$ do overlap, there exists a rational number $d_n>1$ such that $W_n=C_n ^{d_n}$.
Put $U_n:=B_n, V_n:=C_n ^{\lceil d_n /2 \rceil}$, and $w_n:=(d_n +1)/\lceil d_n/2 \rceil$.
Obviously, we have (i) and (iv).
Since $ \lceil d_n/2 \rceil \leq d_n$ and $d_n |C_n|\leq 2 \lceil d_n/2 \rceil |C_n|$, we get $n/2\leq |V_n|\leq n$.
Using (iii) and $|U_n|\leq \kappa n-1$, we can see (ii), (vi), and (vii).
It is immediate that $w_n \geq 3/2$.
Hence, we obtain
\begin{eqnarray*}
\frac{|U_n V_n ^{w_n}|}{|U_n V_n|} & = & 1+\frac{ \lceil (w_n-1)|V_n| \rceil}{|U_n V_n|} \geq 1+ \frac{w_n-1}{|U_n|/|V_n| +1} \\
& \geq & 1+\frac{1/2}{2\kappa +1}= 1+\frac{1}{4\kappa +2}.
\end{eqnarray*}
\end{proof}

\section{Proof of the main results}\label{Proof sec}

\begin{proof}[{\rm Proof of Theorem \ref{main4}}]
Put
\begin{eqnarray*}
\xi _{w,j}:=[0,a_{1,w},\ldots ,a_{\lfloor w^j \rfloor ,w},\overline{b}] \quad \mbox{for } j\geq 1.
\end{eqnarray*}
Since $\xi _w$ and $\xi _{w,j}$ have the same first $(\lfloor w^{j+1} \rfloor -1)$-th partial quotients, while $\lfloor w^{j+1} \rfloor $-th partial quotients are different, we have
\begin{eqnarray*}
|\xi _w -\xi _{w,j}|\asymp |q_{\lfloor w^{j+1} \rfloor}|^{-2}
\end{eqnarray*}
by Lemma \ref{fund} (v) and \ref{lower bound}.
Let $0<\iota <w$ be a real number such that $(w-\iota -2)(w-\iota -1)\geq 2(w+\iota )$ when $p\neq 2$, and $(w-\iota -4)(w-\iota -1)\geq 4(w+\iota )$ when $p=2$.
It is obvious that
\begin{eqnarray*}
|q_{\lfloor w^j \rfloor }|^{w-\iota }\ll |q_{\lfloor w^{j+1} \rfloor }|\ll  |q_{\lfloor w^j \rfloor }|^{w+\iota }
\end{eqnarray*}
for sufficiently large $j$ by Lemma \ref{fund} (iii).
Thus, we have
\begin{eqnarray*}
H(\xi _{w,j})^{-w-\iota }\ll |\xi _w -\xi _{w,j}|\ll  H(\xi _{w,j})^{-w+\iota }
\end{eqnarray*}
for sufficiently large $j$ by Lemma \ref{h and l}.
It follows from Lemma \ref{conj2} and \ref{h and l} that
\begin{eqnarray*}
|\xi _{w,j} -\xi ' _{w,j}|\asymp H(\xi _{w,j})^{-1}.
\end{eqnarray*}
For sufficiently large $j$, we see that
\begin{eqnarray*}
H(\xi _{w,j})\leq H(\xi _{w,j+1})\ll H(\xi _{w,j})^{w+\iota }.
\end{eqnarray*}
It follows from Proposition \ref{Mahler and Koksma} that $w_2 ^{*}(\xi _w)\in [w-\iota -1, w+\iota -1]$ and $w_2 (\xi _w)-w_2 ^{*}(\xi )=1$.
Since $\iota $ is arbitrary, we have $w_2 ^{*}(\xi )=w-1$ and $w_2(\xi )=w$.
\end{proof}

\begin{proof}[{\rm Proof of Theorem \ref{main5}}]
Put
\begin{eqnarray*}
\xi _{w,\eta ,j}:=[0,a_{1,w,\eta },\ldots ,a_{\lfloor w^j\rfloor ,w,\eta },\overline{b,\ldots ,b,d}] \quad \mbox{for } j\geq 1,
\end{eqnarray*}
where the length of period part is $\lfloor \eta w^j\rfloor$.
Since $\lfloor w^j\rfloor +(m_j +1)\lfloor \eta w^j\rfloor >\lfloor w^{j+1}\rfloor$, it follows that $\xi _{w,\eta }$ and $\xi _{w,\eta ,j}$ have the same first $(\lfloor w^{j+1} \rfloor -1)$-th partial quotients, while $\lfloor w^{j+1} \rfloor $-th partial quotients are different.
Thus, we have
\begin{eqnarray*}
|\xi _{w,\eta }-\xi _{w,\eta ,j}|\asymp |q_{\lfloor w^{j+1}\rfloor }|^{-2} \quad \mbox{for } j\geq 1
\end{eqnarray*}
by Lemma \ref{fund} (v) and \ref{lower bound}.
We see that for $j\geq 1$
\begin{eqnarray*}
|\xi _{w,\eta ,j}-\xi _{w,\eta ,j} '|\asymp |q_{\lfloor w^j\rfloor }|^{-2},\quad H(\xi _{w,\eta ,j})\asymp |q_{\lfloor w^j\rfloor }q_{\lfloor w^j\rfloor +\lfloor \eta w^j\rfloor }|
\end{eqnarray*}
by Lemma \ref{conj2} and \ref{h and l}.
Let $0<\iota <\min \{w, 2+\eta \}$ be a real number such that
\begin{eqnarray*}
\left( \frac{2 (w-\iota )}{2+\eta +\iota }-3+\frac{2}{2+\eta }\right) \left( \frac{2(w-\iota )}{2+\eta +\iota }-2+\frac{2}{2+\eta +\iota }\right) \geq 2(w+\iota )\frac{2+\eta +\iota }{2+\eta -\iota } \left( 2-\frac{2}{2+\eta +\iota }\right)
\end{eqnarray*}
when $p\neq 2$, and
\begin{eqnarray*}
\left( \frac{2 (w-\iota )}{2+\eta +\iota }-5+\frac{2}{2+\eta }\right) \left( \frac{2(w-\iota )}{2+\eta +\iota }-2+\frac{2}{2+\eta +\iota }\right) \geq 4(w+\iota )\frac{2+\eta +\iota }{2+\eta -\iota } \left( 2-\frac{2}{2+\eta +\iota }\right)
\end{eqnarray*}
when $p=2$.
It is obvious that
\begin{eqnarray*}
|q_{\lfloor w^j\rfloor }|^{w-\iota }\ll |q_{\lfloor w^{j+1}\rfloor }|\ll |q_{\lfloor w^j\rfloor }|^{w+\iota },\\
|q_{\lfloor w^j\rfloor }|^{1+\eta -\iota }\ll |q_{\lfloor w^j \rfloor +\lfloor \eta w^j \rfloor }|\ll |q_{\lfloor w^j\rfloor }|^{1+\eta +\iota }
\end{eqnarray*}
for sufficiently large $j$.
Hence, we obtain
\begin{eqnarray*}
H(\xi _{w,\eta ,j})^{-\frac{2(w+\iota )}{2+\eta -\iota }}\ll |\xi _{w,\eta }-\xi _{w,\eta ,j}|\ll H(\xi _{w,\eta ,j})^{-\frac{2(w-\iota )}{2+\eta +\iota }},\\
H(\xi _{w,\eta ,j})^{-\frac{2}{2+\eta -\iota }}\ll |\xi _{w,\eta ,j}-\xi _{w,\eta ,j} '|\ll H(\xi _{w,\eta ,j})^{-\frac{2}{2+\eta +\iota }},\\
H(\xi _{w,\eta ,j})\leq H(\xi _{w,\eta ,j+1})\ll H(\xi _{w,\eta ,j})^{(w+\iota )\frac{2+\eta +\iota }{2+\eta -\iota }}
\end{eqnarray*}
for sufficiently large $j$.
It follows from Proposition \ref{Mahler and Koksma} that
\begin{eqnarray*}
w_2 ^{*}(\xi _{w,\eta })\in \left[ \frac{2 w-2-\eta -3\iota }{2+\eta +\iota }, \frac{2 w-2-\eta +3\iota }{2+\eta -\iota } \right] ,\\
w_2(\xi _{w,\eta })-w_2 ^{*}(\xi _{w,\eta })\in \left[ \frac{2}{2+\eta +\iota }, \frac{2}{2+\eta } \right] .
\end{eqnarray*}
Since $\iota $ is arbitrary, we have
\begin{eqnarray*}
w_2 ^{*}(\xi _{w,\eta })=\frac{2 w-2-\eta }{2+\eta },\quad w_2(\xi _{w,\eta })=\frac{2 w-\eta }{2+\eta }.
\end{eqnarray*}
\end{proof}

\begin{proof}[{\rm Proof of Theorem \ref{main1}}]
Applying Lemma \ref{combi}, for $n \geq n_0$, we take finite words $U_n ,V_n$ and a rational number $w_n$ satisfying Lemma \ref{combi} (i)-(v) and (vii).
We define a positive integer sequence $(n_j)_{j\geq 0}$ by $n_{j+1}=2(2\kappa +1)\lceil \log A/\log q\rceil n_j$ for $j\geq 0$.
Put $r_j:=|U_{n_j}|, s_j:=|V_{n_j}|$, and $\tilde{w}_j:=w_{n_j}$ for $j\geq 0$.
By Lemma \ref{combi} (iv), we have $a_{r_j}\neq a_{r_j+s_j}$ for all $j\geq 0$. 
By the assumption and Lemma \ref{fund} (iii), we get $q^n \leq |q_n|\leq A^n$ for all $n\geq 1$.
Therefore, it follows from Lemma \ref{combi} (iii) and (vi) that for $j\geq 0$
\begin{eqnarray}\label{theta}
|q_{r_j} q_{r_j+s_j}|<|q_{r_{j+1}} q_{r_{j+1} s_{j+1}}|\leq |q_{r_j} q_{r_j+s_j}|^{4(2\kappa +1)^2 \left\lceil \frac{\log A}{\log q}\right\rceil \frac{\log A}{\log q}}.
\end{eqnarray}
Put $\alpha _j :=[0,a_1,\ldots ,a_{r_j},\overline{a_{r_j+1},\ldots ,a_{r_j+s_j}}]$ for $j\geq 1$.
By Lemma \ref{height upper}, we obtain
\begin{eqnarray}\label{beta}
H(\alpha _j)\leq |q_{r_j}q_{r_j+s_j}|
\end{eqnarray}
for $j\geq 0$.
Since $\xi $ and $\alpha _j$ have the same first $r_j +\lceil \tilde{w}_j s_j \rceil $-th partial quotients, we have
\begin{eqnarray*}
|\xi -\alpha _j| & \leq & \max \left( \left| \xi -\frac{p_{r_j +\lceil \tilde{w} _j s_j \rceil }}{q_{r_j +\lceil \tilde{w}_j s_j \rceil }}\right| , \left| \alpha _j -\frac{p_{r_j +\lceil \tilde{w}_j s_j \rceil }}{q_{r_j +\lceil \tilde{w}_j s_j \rceil }}\right| \right) \\
& \leq & |q_{r_j +\lceil \tilde{w}_j s_j \rceil }|^{-2} q^{-1} \leq |q_{r_j+s_j}|^{-2} q^{-2(\lceil \tilde{w}_j s_j \rceil -s_j)-1}
\end{eqnarray*}
for $j\geq 0$ by Lemma \ref{fund} (iii) and (v).
By Lemma \ref{combi} (v), we have
\begin{eqnarray*}
q^{2(\lceil \tilde{w}_j s_j \rceil -s_j)+1}\gg q^{\frac{r_j+s_j}{2\kappa +1}}|q_{r_j} q_{r_j +s_j}|^{\frac{\log q}{(4\kappa +2)\log A}}
\end{eqnarray*}
for $j\geq 0$.
From Lemma \ref{conj2}, we deduce that
\begin{eqnarray*}
|q_{r_j}|^2 \leq A^2 \max (|\alpha _j-\alpha ' _j|^{-1},1)
\end{eqnarray*}
for $j\geq 0$.
Hence, we obtain
\begin{eqnarray}\label{delta}
|\xi -\alpha _j|\ll A^2 \max (|\alpha _j -\alpha ' _j|^{-1},1) |q_{r_j}q_{r_j +s_j}|^{-2-\frac{\log q}{(4\kappa +2)\log A}}
\end{eqnarray}
for $j\geq 0$.
Take a real number $\delta $ which is greater than $\Dio ({\bf a})$.
Then $\xi $ and $\alpha _j$ have the same at most $\lceil \delta (r_j +s_j)\rceil $-th partial quotients  for sufficiently large $j$.
By Lemma \ref{lower bound}, we have
\begin{eqnarray}
|\xi -\alpha _j| & \geq & A^{-2} |q_{\lceil \delta (r_j +s_j)\rceil }|^{-2} \gg |q_{r_j} q_{r_j +s_j}|^{-4\delta \frac{(r_j +s_j)\log A}{(2 r_j +s_j)\log q}} \nonumber \\
& \gg & |q_{r_j} q_{r_j +s_j}|^{-4\delta \frac{\log A}{\log q}} \label{rho}
\end{eqnarray}
for sufficiently large $j$.
Applying Proposition \ref{Mahler and Koksma2} with (\ref{theta}), (\ref{beta}), (\ref{delta}), and (\ref{rho}), we obtain
\begin{eqnarray*}
w_2 ^{*}(\xi ) \leq 128(2\kappa +1)^3 \Dio ({\bf a}) \left( \frac{\log A}{\log q}\right) ^4 -1.
\end{eqnarray*}
Thus, we have (\ref{main1.1}) by (\ref{dif}).

Assume that the sequence $(|q_n|^{1/n})_{n\geq 1}$ converges.
Let $M$ be a limit of the sequence $(|q_n|^{1/n})_{n\geq 1}$.
For any $\varepsilon >0$, there exists an integer $n_1$ such that for all $n \geq n_1$,
\begin{eqnarray*}
(M-\varepsilon )^n<|q_n|<(M+\varepsilon )^n.
\end{eqnarray*}
In the same matter as above, we see that
\begin{gather*}
|q_{r_j} q_{r_j+s_j}|<|q_{r_{j+1}} q_{r_{j+1} s_{j+1}}|\leq |q_{r_j} q_{r_j+s_j}|^{4(2\kappa +1)^2 \left\lceil \frac{\log (M+\varepsilon )}{\log (M-\varepsilon )}\right\rceil \frac{\log (M+\varepsilon )}{\log (M-\varepsilon )}}, \\
|q_{r_j} q_{r_j +s_j}|^{-4\delta \frac{\log (M+\varepsilon )}{\log (M-\varepsilon )}}\ll |\xi -\alpha _j|\ll \max (|\alpha _j -\alpha ' _j|^{-1},1) |q_{r_j}q_{r_j +s_j}|^{-2-\frac{\log (M-\varepsilon )}{(4\kappa +2)\log (M+\varepsilon )}},
\end{gather*}
for sufficiently large $j$.
Applying Proposition \ref{Mahler and Koksma2}, we have
\begin{eqnarray*}
w_2 ^{*}(\xi ) \leq 64(2\kappa +1)^3 \Dio ({\bf a})-1.
\end{eqnarray*}
Thus, we have (\ref{main1.2}) by (\ref{dif}).
\end{proof}

\begin{proof}[{\rm Proof of Theorem \ref{main2}}]
From Theorem \ref{Lagrange}, \ref{Mahler lower}, and Proposition \ref{main7}, we have $w_2(\xi )\geq w_2 ^{*}(\xi )\geq 2$.
Without loss of generality, we may assume that $\Dio ({\bf a})>1$.
Take a real number $\delta $ such that $1<\delta <\Dio ({\bf a})$.
For $n\geq 1$, there exist finite words $U_n,V_n $ and a real number $w_n $ such that $U_n V_n ^{w_n}$ is the prefix of ${\bf a}$, the sequence $(|V_n ^{w_n}|)_{n\geq 1}$ is strictly increasing, and $|U_n V_n ^{w_n}|\geq \delta |U_n V_n|$.
Set $r_n :=|U_n|, s_n :=|V_n|$, and $\alpha _n :=[0,a_1,\ldots ,a_{r_n},\overline{a_{r_n+1},\ldots ,a_{r_n+s_n}}]$.
Let $\tilde{M}$ denote an upper bound of $(|q_n|^{1/n})_{n\geq 1}$.
For any $\varepsilon >0$, there exists an integer $n_0$ such that for all $n\geq n_0$,
\begin{eqnarray*}
(m-\varepsilon )^n<|q_n|<(M+\varepsilon )^n.
\end{eqnarray*}
Since $\xi $ and $\alpha _n$ have the same first ($r_n+\lceil w_n s_n\rceil $)-th partial quotients, we obtain
\begin{eqnarray*}
|\xi -\alpha _n|\leq |q_{r_n+\lceil w_n s_n\rceil}|^{-2}<(M+\varepsilon )^{-2(r_n+\lceil w_n s_n\rceil )\frac{\log (m-\varepsilon )}{\log (M+\varepsilon )}}.
\end{eqnarray*}

Assume that the sequences $(r_n)_{n\geq 1}$ and $(s_n)_{n\geq 1}$ are bounded.
Then, for all $n\geq 1$, we have
\begin{eqnarray*}
H(\alpha _n)\leq |q_{r_n}q_{r_n+s_n}|\leq \tilde{M}^{2 r_n+s_n}\leq C,
\end{eqnarray*}
where $C$ is some constant, by Lemma \ref{height upper}.
Therefore, the set $\{ \alpha _n \mid n\geq 1\} $ is finite.
Take a positive integer sequence $(n_i)_{i\geq 1}$ such that $n_i\rightarrow \infty $ as $i\rightarrow \infty $ and $\alpha _{n_1}=\alpha _{n_2}=\cdots $.
Since $(s_n)_{n\geq 1}$ is bounded, we have $w_n \rightarrow \infty $ as $n\rightarrow \infty $.
Hence, we obtain ${\bf a}=U_{n_i}\overline{V_{n_i}}$, which is a contradiction.

We next consider the case that $(r_n)_{n\geq 1}$ is unbounded.
Here, if necessary, taking a subsequence of $(r_n)_{n\geq 1}$, we assume that $(r_n)_{n\geq 1}$ is increasing and $r_1\geq n_0$.
Since $H(\alpha _n)\leq (M+\varepsilon )^{2 r_n +s_n}$ by Lemma \ref{height upper}, we have
\begin{eqnarray*}
|\xi -\alpha _n|\leq H(\alpha _n)^{-\frac{r_n+\lceil w_n s_n\rceil }{r_n+s_n}\frac{\log (m-\varepsilon )}{\log (M+\varepsilon )}}\leq H(\alpha _n)^{-\delta \frac{\log (m-\varepsilon )}{\log (M+\varepsilon )}}.
\end{eqnarray*}
Hence, we obtain (\ref{lDio}).

We consider the case that $(r_n)_{n\geq 1}$ is bounded, $(s_n)_{n\geq 1}$ is unbounded, and $\Dio ({\bf a})$ is finite.
Here, if necessary, taking a subsequence of $(s_n)_{n\geq 1}$, we assume that $(s_n)_{n\geq 1}$ is increasing and $s_1\geq n_0$.
Then, for all $n\geq 1$, we have
\begin{eqnarray*}
H(\alpha _n)\leq \tilde{M}^{r_n} (M+\varepsilon )^{r_n+s_n}<C_1 (M+\varepsilon )^{r_n+s_n},
\end{eqnarray*}
where $C_1$ is some constant.
Therefore, we obtain
\begin{eqnarray*}
|\xi -\alpha _n| & \leq & (C_1 H(\alpha _n)^{-1})^{2\frac{r_n+\lceil w_n s_n\rceil }{r_n+s_n}\frac{\log (m-\varepsilon )}{\log (M+\varepsilon )}}
\leq C_1 ^{2\Dio ({\bf a})} H(\alpha _n)^{-2\delta \frac{\log (m-\varepsilon )}{\log (M+\varepsilon )}}.
\end{eqnarray*}
Hence, we obtain (\ref{lDio}).

We consider the case that $(r_n)_{n\geq 1}$ is bounded, $(s_n)_{n\geq 1}$ is unbounded, and $\Dio ({\bf a})$ is infinite.
Then, for all $n\geq 1$, we have $q^n\leq |q_n|\leq \tilde{M}^n$, which implies $H(\alpha _n)\leq \tilde{M}^{2r_n+s_n}$.
Therefore, in the same matter, we obtain
\begin{eqnarray*}
|\xi -\alpha _n|\leq H(\alpha _n)^{-\delta \frac{\log q}{\log \tilde{M}}}.
\end{eqnarray*}
Hence, we have $w_2 ^{*}(\xi )=+\infty$.

Assume that the sequence $(|a_n|)_{n\geq 1}$ is bounded.
We denote by $A$ its upper bound.
We consider the case that $(r_n)_{n\geq 1}$ is unbounded.
Here, if necessary, taking a subsequence of $(r_n)_{n\geq 1}$, we assume that $(r_n)_{n\geq 1}$ is increasing and $r_1\geq n_0$.
Let $P_n(X)$ be the minimal polynomial of $\alpha _n$.
From Lemma \ref{conj2}, we obtain
\begin{eqnarray*}
|P_n(\xi )| & \leq & H(\alpha _n)|\xi -\alpha _n||\xi -\alpha _n '| \leq A^2 H(\alpha _n)q_{r_n+\lceil w_n s_n\rceil} ^{-2} q_{r_n} ^{-2} \\
& \leq & A^2 H(\alpha _n)^{-2\frac{2 r_n+\lceil w_n s_n\rceil }{2 r_n+s_n}\frac{\log (m-\varepsilon )}{\log (M+\varepsilon )}+1}.
\end{eqnarray*}
Since
\begin{eqnarray*}
\frac{2 r_n+\lceil w_n s_n\rceil }{2 r_n+s_n} & \geq & \frac{r_n+\delta (r_n +s_n)}{2 r_n +s_n}\geq \frac{r_n +s_n/2 +\delta (r_n +s_n /2)}{2 r_n +s_n}
\geq \frac{1+\delta }{2},
\end{eqnarray*}
we obtain (\ref{lDio2}).
For the remaining case, we have (\ref{lDio2}) in the same line of proof of  (\ref{lDio}).
\end{proof}

\appendix
\section{Rational approximation in $\lF _q((T^{-1}))$}\label{Rational sec}

\begin{lem}
Let ${\bf a}=(a_n)_{n\geq 0}$ be a non-ultimately periodic sequence over $\lF _q$.
Set $\xi :=\sum_{n=0}^{\infty}a_n T^{-n}$.
Then we have
\begin{eqnarray}\label{last2}
w_1(\xi )\geq \max (1, \Dio ({\bf a})-1).
\end{eqnarray}
\end{lem}

\begin{proof}
From Theorem \ref{Mahler lower}, we have $w_1(\xi )\geq 1$.
Without loss of generality, we may assume that $\Dio ({\bf a})>1$.
Take a real number $\delta $ such that $1<\delta <\Dio ({\bf a})$.
For $n\geq 1$, there exist finite words $U_n, V_n$ and a real number $w_n$ such that $U_n V_n ^{w_n}$ is the prefix of ${\bf a}$, the sequence $(|V_n ^{w_n}|)_{n\geq 1}$ is strictly increasing, and $|U_n V_n ^{w_n}|\geq \delta |U_n V_n|$.
Put $q_n:=T^{|U_n|}(T^{|V_n|}-1)$.
Then there exists $p_n \in \lF _q[T]$ such that
\begin{eqnarray*}
\frac{p_n}{q_n}=\sum _{k=0}^{\infty }b_k ^{(n)} T^{-k},
\end{eqnarray*}
where $(b_k ^{(n)})_{k\geq 0}$ is the infinite word $U_n\overline{V_n}$ by Lemma 3.4 in \cite{Firicel}.
Since $\xi $ and $p_n/q_n$ have the same first $|U_nV_n ^{w_n}|$-th digits, we obtain
\begin{eqnarray*}
\left| \xi -\frac{p_n}{q_n} \right| \leq |q_n|^{-\delta }.
\end{eqnarray*}
Hence, we have (\ref{last2}).
\end{proof}

The following theorem is an analogue of Th\'eor\`eme 2.1 in \cite{Adamczewski4} and Theorem 1.3 in \cite{Ooto}, and is an extension of Theorem 1.2 in \cite{Firicel}.

\begin{thm}
Let ${\bf a}=(a_n)_{n\geq 0}$ be a non-ultimately periodic sequence over $\lF _q$.
Set $\xi :=\sum_{n=0}^{\infty}a_n T^{-n}$.
Assume that there exist integers $n_0\geq 1$ and $\kappa \geq 2$ such that for all $n\geq n_0$,
\begin{eqnarray*}
p({\bf a}, n) \leq \kappa n.
\end{eqnarray*}
If the Diophantine exponent of ${\bf a}$ is finite, then we have
\begin{eqnarray}\label{last}
w_1(\xi )\leq 8(\kappa +1)^2(2\kappa +1)\Dio ({\bf a})-1.
\end{eqnarray}
\end{thm}

\begin{proof}
For $n\geq n_0$, take finite words $U_n ,V_n$ and a rational number $w_n$ satisfying Lemma \ref{combi} (i)-(vi).
Put $q_n:=T^{|U_n|}(T^{|V_n|}-1)$.
Then there exists $p_n \in \lF _q[T]$ such that
\begin{eqnarray*}
\frac{p_n}{q_n}=\sum _{k=0}^{\infty }b_k ^{(n)} T^{-k},
\end{eqnarray*}
where $(b_k ^{(n)})_{k\geq 0}$ is the infinite word $U_n\overline{V_n}$ by Lemma 3.4 in \cite{Firicel}.
Since $\xi $ and $p_n/q_n$ have the same first $|U_nV_n ^{w_n}|$-th digits, we obtain
\begin{eqnarray*}
\left| \xi -\frac{p_n}{q_n} \right| \leq |q_n|^{-1-\frac{1}{4\kappa +2}}.
\end{eqnarray*}
Take a real number $\delta $ which is greater than $\Dio ({\bf a})$.
Note that $\delta >1.$
By the definition of Diophantine exponent, there exists an integer $n_1\geq n_0$ such that for all $n\geq n_1$
\begin{eqnarray*}
\left| \xi -\frac{p_n}{q_n} \right| \geq |q_n|^{-\delta }.
\end{eqnarray*}
We define a positive integer sequence $(n_j)_{j\geq 1}$ by $n_{j+1}=2(\kappa +1)n_j$ for $j\geq 1$.
It follows from Lemma \ref{combi} (iii) and (vi) that for $j\geq 1$
\begin{eqnarray*}
|q_{n_j}|<|q_{n_{j+1}}|\leq |q_{n_j}|^{4(\kappa +1)^2}.
\end{eqnarray*}
Thus, by Lemma 3.2 in \cite{Firicel}, we obtain (\ref{last}).
\end{proof}

Consequently, the following result holds.

\begin{cor}
Let ${\bf a}=(a_n)_{n\geq 0}$ be a non-ultimately periodic sequence over $\lF _q$.
Set $\xi :=\sum_{n=0}^{\infty}a_n T^{-n}$.
Then the Diophantine exponent of ${\bf a}$ is finite if and only if $\xi $ is not a $U_1$-number.
\end{cor}

\subsection*{Acknowledgements}
I would like to express my gratitude to Prof.~Shigeki Akiyama for improving the language and structure of this paper.

\end{document}